\documentclass[a4paper]{article}
\pdfoutput=1
\usepackage[draft]{fixme}
\usepackage{amsmath, amsfonts, amssymb}
\usepackage{amsthm}
\usepackage{graphicx}
\usepackage{inputenc, fontenc}
\usepackage{mathtools}
\usepackage{authblk}

\usepackage{natbib}
    \newcommand{\R}{{\mathbb R}}
    \newcommand{\Rn}{{\mathbb R}^n}
    \newcommand{\N}{{\mathbb N}}

    \newcommand{\Q}{\mathbb{Q}}
    \newcommand{\Tensor}{\mathbb{T}}

    \newcommand{\BE}{\mathbb{E}}    
	\newcommand{\Prob}{\mathbb{P}}

    \newcommand{\cE}{\mathcal{E}}
    \newcommand{\cL}{\mathcal{L}}
    \newcommand{\K}{{\mathcal K}}   
    
    \newcommand{\B}{\mathcal{B}}
    \newcommand{\Haus}{\mathcal{H}}
    \newcommand{\A}{\mathcal{A}}

    \newcommand{\Lp}{L^{\perp}}

\newcommand{\cLo}{\mathcal{L}_1^n}
\newcommand{\cLL[1]}[r]{\mathcal{L}^{L}_{#1}}
\newcommand{\cLLn[1]}[r]{\mathcal{L}^{L_0}_{#1}}

\newcommand{\En}{\mathcal{E}^n}

\newcommand{\Eorigo}{\pi(E)}
\newcommand{\Horigo}{\pi (H)}

\newcommand{\1}{\textbf{1}}

\newcommand\IP[2]{\langle {#1} , {#2} \rangle}

\newcommand{\integral[1]}[s]{\int_{\mathcal{E}^n_1}\Phi^{(E)}_{0,0,#1}(K\cap E)\, \mu_1^n(dE)}

\newcommand{\T[1]}[2k]{\Phi_{n-1,0,#1}(K)}
\newcommand{\MT[1]}[s]{\Phi_{0,0,#1}^{(E)}(K \cap E)}

\newcommand{\ch}{\hat{\varphi}_{ij}(K \cap E)}

\newcommand{\oT[1]}[s]{\overline{\Phi}_{n-1,0,#1}(X)}
\newcommand{\oMT[1]}[s]{\overline{\Phi}_{0,0,#1}^{(L)}(X \cap L)}
\newcommand{\jT[1]}[s]{\Phi_{j,0,#1}(K)}
\newcommand{\joT[1]}[s]{\overline{\Phi}_{j,0,#1}(X)}
\newcommand{\norm}{\|}

    \newtheorem{proposition}{Proposition}[section]
    \newtheorem{theorem}[proposition]{Theorem}
    \newtheorem{corollary}[proposition]{Corollary}
    \newtheorem{lemma}[proposition]{Lemma}
    
    \newtheorem{example}[proposition]{Example}
    \newtheorem{definition}[proposition]{Definition}

    \newenvironment{proofof}[1][\proofname]{%
  \proof[Proof of #1]%
}{\endproof}

\title{Surface tensor estimation from linear sections}
\author[1]{Astrid Kousholt\footnote{E-mail: kousholt@imf.au.dk}}
\author[1]{Markus Kiderlen}
\author[2]{Daniel Hug}
\affil[1]{Department of Mathematics, Aarhus University, Denmark}
\affil[2]{Department of Mathematics, Karlsruhe Institute of Technology (KIT), Germany}
\date{}

\begin{document}

\maketitle

\begin{abstract} 
From Crofton's formula for Minkowski tensors we derive stereological estimators of translation invariant surface tensors of convex bodies in the $n$-dimensional Euclidean space. The estimators are based on one-dimensional linear sections. In a design based setting we suggest three types of estimators. These are based on isotropic uniform random lines, vertical sections, and non-isotropic random lines, respectively. Further, we derive estimators of the specific surface tensors associated with a stationary process of convex particles in the model based setting.
\end{abstract}
\vspace{2mm}

\section{Introduction}
In recent years, there has been an increasing interest in  Minkowski tensors as descriptors of morphology and shape of spatial structures of physical systems. 
For instance, they have been established as robust and versatile measures of anisotropy in \cite{Beisbart2002,SM10,Schroder-Turk2013}. In addition to the applications in materials science, \cite{Beisbart} indicates that the Minkowski tensors lead to a putative taxonomy of neuronal cells. From a pure theoretical point of view, Minkowski tensors are, likewise, interesting. This is illustrated by Alesker's characterization theorem  \cite{Alesker1999}, stating that the basic tensor valuations (products of the Minkowski tensors and powers of the metric tensor) span the space of tensor-valued valuations satisfying some natural conditions. 

This paper presents estimators of certain Minkowski tensors from measurements in one-dimensional flat sections of the underlying geometric structure. We restrict attention to translation invariant Minkowski tensors of convex bodies, more precisely, to those that are derived from the top order surface area measure; see Section \ref{prelim} for a definition. 
As usual, the estimators are derived from an integral formula, namely the Crofton formula for 
Minkowski tensors. We adopt the classical setting where the sectioning space is affine and integrated with respect to the motion invariant measure. Rotational Crofton formulae where the sectioning space is a linear subspace and the rotation invariant measure on the corresponding Grassmannian is used, are established in \cite{A-C2013}. The latter  formulae were the basis for local stereological estimators of certain Minkowski tensors
 in \cite{Jensen2013} (for $j \in \{1, \dots, n-1\}, s,r \in \{0,1\}$ and $j=n, s=0, r \in \N$ in the notation of \eqref{volumeTensor} and \eqref{Mtensor}, below).

Kanatani \cite{Kanatani1984,Kanatani1984a} was apparently the first to use tensorial quantities to detect and analyse structural anisotropy via basic stereological principles. He expresses the expected number $N(m)$ of intersections per unit length of a probe with a test line of given direction $m$ as the cosine transform of the spherical distribution density $f$ of the surface of the given probe in $\R^n$ for $n=2,3$. The relation between $N$ and $f$ is studied by expanding $f$ into spherical harmonics and by using the fact that these are eigenfunctions of the cosine transform. In order to express his results  independently of a particular coordinate system,  Kanatani uses tensors. For a fixed $s$, he considers the vector space $V_s$ of all symmetric tensors spanned by the elementary tensor products $u^{\otimes s}$ of vectors $u$  from the unit sphere $S^{n-1}$. Let $\hat T$ denote the deviator part (or trace-free part) of some symmetric tensor $T$. The tensors  $\widehat{(u^{\otimes k})}$, for $k\le s$ and $u\in S^{n-1}$, then span $V_s$ and  the components of $\widehat{(u^{\otimes k})}$ with respect to an orthonormal basis of $\R^n$ are spherical harmonics of degree $k$, when considered as functions of $u$. Hence, $u\mapsto \widehat{(u^{\otimes k})}$ is an eigenfunction of the cosine transform (Kanatani calls it `Buffon transform'), which in fact is the underlying integral transform when considering Crofton integrals with lines, as we shall see below in \eqref{eq1}.  In \cite{Kanatani1984b,Kanatani1985}, he suggests to use these `fabric tensors' to detect surface motions and the anisotropy of the crack distribution in rock.

General Crofton formulas in $\R^n$ with arbitrary dimensional flats and for general Minkowski tensors (defined in \eqref{Mtensor}) of arbitrary rank are given in \cite{Hug}. Theorem \ref{thm} is a special case of one of these results, for translation invariant surface tensors and one-dimensional sections, that is, sections with lines. 
In comparison to \cite{Hug}, we get simplified  constants in the case considered and obtain this result by an elementary independent proof. In contrast to Kanatani's approach, our proof does not rely on spherical harmonics. Here we focus on relative Crofton formulas in which the Minkowski tensors of the  
sections with lines are calculated relative to the section lines and not in the ambient space (Crofton formulas of the second type may be called extrinsic Crofton formulas). A quite general investigation of integral geometric formulas for translation invariant Minkowski tensors, including extrinsic Crofton formulas, is provided in \cite{BernigHug}.

In Theorem \ref{thm} we prove that the relative Crofton integral for tensors of arbitrary even rank $s$ of sections with lines is equal to a linear combination of surface tensors of rank at most $s$. From this we deduce by the inversion of a linear system that any translation invariant surface tensor of even rank $s$ can be expressed as a Crofton integral. The involved measurement functions then are linear combinations of relative tensors of rank at most $s$. This implies that the measurement functions only depend on the convex body through the Euler characteristic of the intersection of the convex body and the test line. 

Our results do not allow to write surface tensors of odd rank as Crofton integrals based on sections with lines. This drawback is not a result of our method of proof. Indeed, apart from the trivial case of tensors of rank one, there does not exist a translation invariant or a bounded  measurement function that expresses a surface tensor of odd rank as a Crofton integral; see Theorem \ref{s ulige} for a precise statement of this fact.

In Section \ref{sec estimation} the integral formula for surface tensors of even rank is transferred to stereological formulae in a design based setting. Three types of unbiased estimators are discussed. Section \ref{IUR} describes an estimator based on isotropic uniform random lines. Due to the structure of the measurement function, it suffices to observe whether the test line hits or misses the convex body in order to estimate the surface tensors. However, the resulting estimators possess some unfortunate statistical properties. In contrast to the surface tensors of full dimensional convex bodies, the estimators are not positive definite. For convex bodies, which are not too eccentric (see \eqref{posdefcondition}), this problem is solved by using $n$ orthogonal test lines in combination with a measurement of the projection function of order $n-1$ of the convex body.   

In applications it might be inconvenient or even impossible to construct the isotropic uniform random lines, which are necessary for the use of the estimator described above. Instead, it might be a possibility to use vertical sections; see Definition \ref{VUR}. A combination of Crofton's formula and a result of  Blaschke-Petkantschin type allows us to formulate a vertical section estimator. The estimator, which is discussed in Section \ref{sec VUR}, is based on two-dimensional vertical flats. 
 
The third type of estimator presented in the design based setting is based on non-isotropic linear sections; see Section \ref{Sec noniso}. For a fixed convex body in $\R^2$ there exists a density for the distribution of test line directions in an importance-sampling approach that leads to minimal variance of the non-isotropic estimator, when we consider one component of a rank 2 tensor, interpreted as a matrix. In practical applications, this density is not accessible, as it depends on the convex body, which is typically unknown. However, there does exist a density independent of the underlying convex body yielding an estimator with smaller variance than the estimator based on isotropic uniform random lines. If \emph{all} components of the tensor are sought for, the non-isotropic approach requires three test lines, as two of the four components of a rank 2 Minkowski tensor coincide due to symmetry. It should be avoided to use a density suited for estimating one particular component of the tensor to estimate any other component, as this would increase variance of the estimator. In this situation, however, a smaller variance can be obtained by applying an estimator based on \emph{three} isotropic random lines (each of which can be used for the estimation of \emph{all} components of the tensor).  

In Section \ref{SecModel} we turn to a model-based setting. We discuss estimation of the \textit{specific (translation invariant) surface tensors} associated with a stationary process of convex particles; see \eqref{defmeansurface} for a definition. In \cite{RSRS06} the problem of estimating the area moment tensor (rank $2$) associated with a stationary process of convex particles via planar sections is discussed. We consider estimators of the specific surface tensors of arbitrary even rank based on one-dimensional linear sections. Using the Crofton formula for surface tensors, we derive a rotational Crofton formula for the specific surface tensors. Further, the specific surface tensor of rank $s$ of a stationary process of convex particles is expressed as a rotational average of a linear combination of specific tensors of rank at most $s$ of the sectioned process.

\section{Preliminaries} \label{prelim}
We work in the $n$-dimensional Euclidean vector space $\Rn$ with inner product $ \IP{\cdot}{\cdot} $ and induced norm $ \norm \cdot \norm $. Let $B^n:=\{x \in \Rn \mid \norm x \norm \leq 1\}$ be the unit ball and $ S^{n-1}:=\{x \in \Rn \mid \norm x \norm =1\} $  the unit sphere in $\Rn$. By $\kappa_n$ and $\omega_{n}$ we denote the volume and the surface area of $ B^n $, respectively. The Borel $\sigma$-algebra of a topological space $X$ is denoted by $\B(X)$.  Further, let $\lambda$ denote the $n$-dimensional Lebesgue measure on $\R^n$, and for an affine subspace $E$ of $\R^n$, let $\lambda_E$ denote the Lebesgue measure defined on $E$. The $k$-dimensional Hausdorff measure is denoted by $\Haus^k$. For $A \subseteq \R^n$, let  $\dim A$  be the dimension of the affine hull of $A$.  

Let $\Tensor^p$ be the vector space of symmetric tensors of rank $p$ over $\R^n.$ For symmetric tensors $a\in \Tensor^{p_1}$ and $b\in \Tensor^{p_2}$, let $ab\in \Tensor^{p_1+p_2}$ denote the symmetric tensor product of $a$ and $b$. Identifying $x \in \R^n$ with the rank 1 tensor $z \mapsto \IP{z}{x}$, we write  $x^p \in \Tensor^p$ for the $p$-fold symmetric tensor product of $x$. The metric tensor $Q \in \Tensor^2$ is defined by $Q(x,y)=\IP{x}{y}$ for $x,y \in \R^n$, and for a linear subspace $L$ of $\R^n$, we define $Q(L) \in \Tensor^2$ by $Q(L)(x,y)=\IP{p_L(x)}{p_L(y)}$, where $p_L\colon \R^n \rightarrow L$ is the orthogonal projection on $L$.

As general references on convex geometry and Minkowski tensors, we use \cite{Schneider93} and \cite{Hug}. Let $\K^n$ denote the set of convex bodies (that is,  compact, convex sets) in $\R^n$.  In order to define the Minkowski tensors, we introduce the support measures $\Lambda_0(K, \cdot), \dots, \Lambda_{n-1}(K, \cdot)$ of a non-empty, convex body $ K \in\K^n$. Let $ p(K,x) $ be the metric projection of $x \in \Rn$ on a non-empty convex body $ K $, and define $ u(K,x):=\frac{x-p(K,x)}{\norm x-p(K,x)\norm} $ for $x \notin K$. For $ \epsilon >0 $ and a Borel set $\nobreak{A \in \B(\R^n \times S^{n-1}})$, the Lebesgue measure of the local parallel set
\begin{equation*}
M_\epsilon(K,A):=\{x \in (K+\epsilon B^n) \setminus K \mid (p(K,x),u(K,x)) \in A \}
\end{equation*}
of $K$ is a polynomial in $\epsilon$, hence
\begin{equation*}
\lambda(M_\epsilon(K,A))=\sum_{k=0}^{n-1}\epsilon^{n-k}\kappa_{n-k}\Lambda_k(K,A).
\end{equation*}
This local version of the Steiner formula defines the support measures $ \Lambda_0(K, \cdot), \allowbreak \dots, \Lambda_{n-1}(K,\cdot) $ of a non-empty convex body $ K\in\K^n$. If $K=\emptyset$, we define the support measures to be the zero measures. The intrinsic volumes $ V_0(K), \dots, V_{n-1}(K)$ of $ K $ appear as total masses of the support measures, $V_j(K)=\Lambda_j(K,\Rn\times S^{n-1})$ for $j=0, \dots, n-1$. Furthermore, the area measures $ S_0(K, \cdot), \dots, S_{n-1}(K,\cdot) $ of $ K $ are rescaled projections of the corresponding support measures on the second component. More explicitly, they are given by 
\begin{equation*}
\binom{n}{j}S_j(K,\omega)=n \kappa_{n-j}\Lambda_j(K,\Rn \times \omega)
\end{equation*}
for $ \omega \in \B(S^{n-1}) $ and $j=0, \dots, n-1$. 

For a non-empty convex body $K\in\K^n$, $r,s \in \N_0$, and $ j \in \{0,1,\dots, n-1\} $, we define the \textit{Minkowski tensors} as
\begin{equation}\label{Mtensor}
\Phi_{j,r,s}(K) :=\frac{\omega_{n-j}}{r!s! \omega_{n-j+s}} \int_{\R^n \times S^{n-1}} x^r u^s\, \Lambda_j (K,d(x,u))
\end{equation}
and
\begin{equation} \label{volumeTensor}
\Phi_{n,r,0}(K):=\frac{1}{r!}\int_{K}x^r \,\lambda(dx).
\end{equation}
The definition of the Minkowski tensors is extended by letting $\Phi_{j,r,s}(K)=0$, if $j \notin\{0,1,\dots, n\}$, or if $r$ or $s$ is not in $\N_0$, or if $j=n$ and $s \neq 0$. For $j=n-1$, the tensors \eqref{Mtensor} are called surface tensors. In the present work, we only consider translation invariant surface tensors which are obtained for $r=0$. In \cite{Hug} the functions $Q^m\Phi_{j,r,s}$ with $m,r,s \in \N_0$ and either $ j \in \{0,\dots, n-1\} $ or $ (j,s)=(n,0) $ are called the basic tensor valuations. 

For $k \in \{1, \dots,n \}$, let $\cL^n_k$ be the set of $k$-dimensional linear subspaces of $\R^n$, and let $\En_k$ be the set of $ k $-dimensional affine subspaces of $\R^n$. For $L \in \cL^n_k$, we write $L^\perp$ for the orthogonal complement of $L$. For $E \in \En_k$,  let $\Eorigo$ denote the linear subspace in $\cL^n_k$ which is parallel to $E$, and we define $ E^\perp :=\Eorigo^\perp$. The sets $\cL^n_k$ and $\cE^n_k$ are endowed with their usual topologies and Borel $\sigma$-algebras. Let $\nu^n_k$ denote the unique rotation invariant probability measure on $\cL^n_k$, and let $\mu_k^n$ denote the unique motion invariant measure on $\En_k$ normalized so that $\nobreak{\mu^n_k(\{E \in \En_k \vert E \cap B^n \neq \emptyset\})}=\kappa_{n-k}$ (see, e.g.,  \cite{Weil}).

If $K \in \K^n$ is non-empty and contained in an affine subspace $E \in \En_k$,  for some $k \in \{1, \dots, n\}$, then the Minkowski tensors can be evaluated in this subspace. For a linear subspace $ L \in \cL_k^n $, let $ \pi_L \colon S^{n-1} \setminus L^\perp \rightarrow  L \cap S^{n-1}$ be given by 
\begin{equation*}
\pi_L(u):=\frac{p_L(u)}{\norm p_L(u)\norm}.
\end{equation*}
Then we define the $ j $th support 
measure $ \Lambda_j^{(E)}(K, \cdot) $  of $ K $ relative to $ E $ as the image measure of the restriction of $\Lambda_j(K, \cdot)$  to $ \Rn \times (S^{n-1} \setminus E^\perp) $ under the mapping $ \Rn \times (S^{n-1} \setminus E^\perp) \rightarrow \Rn \times (\Eorigo \cap S^{n-1})  $ given by $(x,u) \mapsto (x,\pi_{\Eorigo}(u))$.  

For a non-empty convex body $K \in \K^n$, contained in an affine subspace $ E \in \En_k $,  for some $k \in \{1, \dots, n\}$, we define
\begin{equation*}
\Phi_{j,r,s}^{(E)}(K) :=\frac{\omega_{k-j}}{r!s! \omega_{k-j+s}} \int_{E \times (S^{n-1} \cap \Eorigo) } x^r u^s \, \Lambda^{(E)}_j (K,d(x,u))
\end{equation*}
for $ r,s \in \N_0 $ and $ j \in \{0,\dots, k-1\} $, and
\begin{equation*}
\Phi_{k,r,0}^{(E)}(K):=\frac{1}{r!}\int_{K} x^r \,
 \lambda_E(dx).
\end{equation*}
As before, the definition is extended by letting $\Phi^{(E)}_{j,r,s}(K)=0$ for all other choices of $j,r$ and $s$, and for $K=\emptyset$.

In \cite{Hug},  Crofton integrals of the form
\begin{equation*}
\int_{\En_k} \Phi_{j,r,s}^{(E)}(K \cap E) \, \mu_k^n (dE),
\end{equation*}
where $K \in \K^n$, $r,s \in \N_0$ and $0 \leq j \leq k \leq n-1$,  are expressed as linear 
combinations of the basic tensor valuations. When $j=k$ the integral formula becomes
\begin{equation}\label{j=k}
\int_{\En_k} \Phi_{k,r,s}^{(E)}(K \cap E) \, \mu_k^n (dE)=
\begin{cases}
\Phi_{n,r,0}(K) & \text{if } s=0, \\
0 & \text{otherwise}, 
\end{cases}
\end{equation}
see \cite[Theorem 2.4]{Hug}. In the case where $j < k$, the formulas become lengthy with coefficients in the linear combinations that are difficult to evaluate, see \cite[Theorem 2.5 and 2.6]{Hug}. In the following, we are interested in using the integral formulas for the estimation of the surface tensors, and therefore we need more explicit integral formulas. We only treat the special case where $k=1$, that is, we consider integrals of the form
\begin{equation*}
\int_{\En_1} \Phi_{j,r,s}^{(E)}(K \cap E) \, \mu_1^n (dE).
\end{equation*}
Since $\dim(E)=1$, the tensor $\Phi_{j,r,s}^{(E)}(K)$ is by definition the zero function when $j>1$, so the only non-trivial cases are $j=0$ and $j=1$. When $j=1$ formula \eqref{j=k} gives a simple expression for the integral. In the case where $j=0$ and $r=0$, we provide an independent and elementary proof of the integral formula, which also leads to explicit and fairly simple constants.

\section{Linear Crofton formulae for tensors} \label{CroftonSec}

We start with the main result of this section, which provides a linear Crofton formula relating an average of tensor valuations defined relative to varying  section lines to a linear combination of surface tensors.

\begin{theorem} \label{crofton-like}
Let $K \in \K^n$. If $s \in \N_0$ is even, then
\begin{equation}\label{thm}
\int_{\mathcal{E}^n_1}\Phi^{(E)}_{0,0,s}(K\cap E)\, \mu_1^n(dE) = \frac{2 \omega_{n+s+1}}{ \pi s!\omega_{s+1}^2\omega_n} \sum_{k=0}^{\frac{s}{2}} c_{k}^{(\frac{s}{2})} Q^{\frac{s}{2}-k}  \Phi_{n-1,0,2k}(K),
\end{equation}
with constants
\begin{equation}
c_{k}^{(m)}=(-1)^{k} \binom{m}{k} \frac{(2k)! \, \omega_{2k+1}}{1-2k} 
\end{equation}
for $m \in \N_0$ and $k=0, \dots, m$.

For odd $s \in \N_0$ the Crofton integral on the left-hand side is zero.
\end{theorem}

Before we give a proof of Theorem $\ref{crofton-like}$, let us consider the measurement function $\MT[s]$ on the left-hand side of \eqref{thm}. Let $k \in \{1, \dots, n\}$. Slightly more general than in \eqref{thm}, we choose $s \in \N_0$ and  $E \in \cE_k^n$. Then 
\begin{equation*}
\Phi^{(E)}_{0,0,s}(K\cap E)
=\frac{1}{s!\omega_{k+s}}\int_{S^{n-1}\cap \Eorigo}u^s\, \Haus^{k-1}(du)\, V_0(K\cap E),
\end{equation*}
since the surface area measure of order 0 of a non-empty set is up to a constant the invariant measure on the sphere. From calculations equivalent to \cite[(24)-(26)]{TVCB} (or from a special case of Lemma 4.3 in \cite{Hug}) we get that 
\begin{equation}\label{metric}
\int_{S^{n-1}\cap \Eorigo}u^s\, \mathcal{H}^{k-1}(du)
=\begin{cases}
2\frac{\omega_{s+k}}{\omega_{s+1}}Q(\Eorigo)^{\frac{s}{2}} & \text{if $s$ is even}, \\
0      & \text{if $s$ is odd}.
\end{cases}
\end{equation}
Hence 
\begin{equation}\label{Simpel form for T}
\Phi^{(E)}_{0,0,s}(K\cap E)
=\frac{2}{s! \omega_{s+1}}\cdot Q(\Eorigo)^{\frac{s}{2}}V_0(K\cap E), 
\end{equation}
when $s$ is even, and $\Phi^{(E)}_{0,0,s}(K\cap E)=0$ when $s$ is odd. This implies that the Crofton integral in \eqref{thm} is zero for odd $s$, and the tensors $\T[s]$ are hereby not accessible in this situation. This is even true for more general measurement functions; see Theorem \ref{s ulige}. To show Theorem \ref{crofton-like} we can restrict to even $s$ from now on.

In the proof of Theorem \ref{crofton-like} we use the following identity for binomial sums.

\begin{lemma}\label{binomial}
Let $m,n \in \N_0$. Then
\begin{equation*}
\sum_{j=0}^m(-1)^j \frac{\binom{2n}{2j}\binom{n-j}{m-j}}{\binom{n-\frac{1}{2}}{j}}
=
\frac{\binom{n}{m}}{1-2m}.
\end{equation*}
\end{lemma}
Lemma \ref{binomial} can be proven by using the identity
\begin{equation}\label{hjaelpebinomial}
\sum_{j=0}^k (-1)^j \frac{\binom{2n}{2j} \binom{n-j}{m-j}}{\binom{n-\frac{1}{2}}{j}}
=
\frac{(-1)^k(2k+1)\binom{2n}{2(k+1)} \binom{n-k-1}{m-k-1}}{(2m-1)\binom{n-\frac{1}{2}}{k+1}}-\frac{\binom{n}{m}}{(2m-1)},
\end{equation}
where $n, k \in \N_0$, and $m \in \N$ such that $k<m$. Identity \eqref{hjaelpebinomial} follows by induction on $k$.

\begin{proofof}[Theorem \ref{crofton-like}]
Let $K \in \K^n$ and let $s \in \N_0$ be even. If $n=1$, formula \eqref{thm} follows from the identity
\begin{equation}\label{binom2}
\sum_{j=0}^m (-1)^j\frac{\binom{m}{j}}{1-2j}=\frac{\sqrt{\pi} \,\Gamma(m+1)}{\Gamma(m+\frac{1}{2})}
\end{equation}
with $m=\frac{s}{2}$. The left-hand side of \eqref{binom2} is a sum of alternating terms of the same form as the right-hand side of the binomial sum in Lemma \ref{binomial}. Using Lemma \ref{binomial} and then changing the order of summation yields \eqref{binom2}. 

Now assume that $n\geq2$. Using \eqref{Simpel form for T} we can rewrite the integral as
\begin{align*}
&\int_{\mathcal{E}^n_1}\Phi^{(E)}_{0,0,s}(K\cap E)\, \mu_1^n(dE)\\
&\qquad =\frac{2}{s!\omega_{s+1}} 
  \int_{\cL_1^n} Q(L)^{\frac{s}{2}} \int_{\Lp} V_0(K \cap (L+x)) \,\lambda_{\Lp}(dx)\,\nu^n_1(dL)
\\
&\qquad =\frac{2}{s!\omega_{s+1}\omega_n} \int_{S^{n-1}}u^sV_{n-1}(K\mid u^\perp)\, \Haus^{n-1}(du),
\end{align*}
by the convexity of $K$ and an invariance argument for the second equality.  Cauchy's projection formula (see, e.g., \cite[(A.43)]{Gardner}) and Fubini's theorem then imply that
\begin{align}\label{eq1}
&\int_{\mathcal{E}^n_1}\Phi^{(E)}_{0,0,s}(K\cap E)\, \mu_1^n(dE)\nonumber \\
&\qquad =\frac{1}{s!\omega_{s+1}\omega_n}  
 \int_{S^{n-1}}\int_{S^{n-1}} u^s |\langle u,v\rangle|\, \mathcal{H}^{n-1}(du)\, S_{n-1}(K,dv).
\end{align}

We now fix $v \in S^{n-1}$ and simplify the inner integral by introducing spherical coordinates (see, e.g, \cite{SH66}). Then
\begin{align*}
&\int_{S^{n-1}}u^s |\langle u,v\rangle|\, \mathcal{H}^{n-1}(du)\\
&\qquad =\int_{-1}^1\int_{S^{n-1}\cap v^\perp}(1-t^2)^{\frac{n-3}{2}}(tv+\sqrt{1-t^2}w)^s|t|\,\Haus^{n-2}(dw) \, dt
\\
&\qquad=\sum_{j=0}^s \binom{s}{j}v^j \int_{-1}^1 (1-t^2)^{\frac{n-3}{2}}t^j\sqrt{1-t^2}^{s-j} |t| \,dt \int_{S^{n-1}\cap v^\perp} w^{s-j}\,\Haus^{n-2}(dw).
\end{align*}
The integral with respect to $t$ is zero if $j$ is odd. If $j$ is even, then it is equal to the beta integral
\begin{equation*}
B\bigg(\frac{j+2}{2},\frac{n+s-j-1}{2}\bigg)
=\frac{2\omega_{n+s+1}}{\omega_{j+2}\,\omega_{n+s-j-1}}.
\end{equation*}
Hence, since $s$ is even, we conclude from \eqref{metric} that
\begin{align*}
&\int_{S^{n-1}}u^s |\langle u,v\rangle|\, \Haus^{n-1}(du)=4 \omega_{n+s+1} \sum_{j=0}^{\frac{s}{2}} \binom{s}{2j}v^{2j} \frac{1}{\omega_{2j+2} \, \omega_{s-2j+1}}Q(v^{\perp})^{\frac{s-2j}{2}}
\\
&\qquad =4 \omega_{n+s+1} \sum_{j=0}^{\frac{s}{2}} \sum_{i=0}^{\frac{s}{2}-j} (-1)^{i} \binom{s}{2j} \binom{\frac{s}{2}-j}{i} \frac{1}{\omega_{2j+2} \, \omega_{s-2j+1}}  Q^{\frac{s}{2}-j-i}v^{2(i+j)},
\end{align*}
where we have used that $Q(v^{\perp})=Q-v^2$. Substituting this into \eqref{eq1} and by the definition of $\T[2(i+j)]$, we obtain that
\begin{equation}\label{S}
\int_{\mathcal{E}^n_1}\Phi^{(E)}_{0,0,s}(K\cap E)\, \mu_1^n(dE)
=\frac{4 \omega_{n+s+1}}{s!\omega_{s+1}\omega_n} \, S,
\end{equation}
where
\begin{equation*}
S= \sum_{j=0}^{\frac{s}{2}} \sum_{i=0}^{\frac{s}{2}-j} (-1)^{i} \binom{s}{2j} \binom{\frac{s}{2}-j}{i} \frac{(2(i+j))!\omega_{2(i+j)+1}}{\omega_{2j+2} \, \omega_{s-2j+1}}  Q^{\frac{s}{2}-j-i}  \Phi_{n-1,0,2(i+j)}(K).
\end{equation*}
Re-indexing and changing the order of summation, we arrive at
\begin{align*}
S&=
\frac{\Gamma(\frac{s}{2}+\frac{1}{2})}{4 \pi^{\frac{s+3}{2}}} \sum_{k=0}^{\frac{s}{2}} (-1)^k (2k)! \omega_{2k+1} Q^{\frac{s}{2}-k}\Phi_{n-1,0,2k}(K)
\\
&\qquad \times \sum_{j=0}^{k}(-1)^{j}\binom{s}{2j}\binom{\frac{s}{2}-j}{k-j} \binom{\frac{s-1}{2}}{j}^{-1}
\\
&=
 \frac{1}{2 \pi \omega_{s+1}}\sum_{k=0}^{\frac{s}{2}} (-1)^k  \binom{\frac{s}{2}}{k}\frac{(2k)! \, \omega_{2k+1}}{1-2k} Q^{\frac{s}{2}-k}\Phi_{n-1,0,2k}(K),
\end{align*}
where we have used Lemma \ref{binomial} with $n=\frac{s}{2}$ and $m=k$.
\end{proofof}

Setting $s=2$ we immediately get the following corollary.
\begin{corollary}
Let $K \in \K^n$. Then
\begin{equation*}
\int_{\mathcal{E}^n_1}\Phi^{(E)}_{0,0,2}(K\cap E)\, \mu_1^n(dE) 
=a_n\bigg(\Phi_{n-1,0,2}(K) + \frac{1}{4\pi} Q V_{n-1}(K)\bigg),
\end{equation*}
where
\begin{equation*}
a_n=\frac{\Gamma(\frac{n}{2})}{2\Gamma(\frac{n+3}{2})\sqrt{\pi}}.
\end{equation*}
\end{corollary}

The Crofton formula in Theorem \ref{crofton-like} expresses the integral of the measurement function $\MT[s]$ as a linear combination of certain surface tensors of $K \in \K^n$. This could, in principle, be used to obtain unbiased stereological estimators of the linear combinations. However, it is more natural to ask what measurement one should use in order to obtain $\T[s]$ as a Crofton-type integral. For even $s$ the tensor $\T[s]$ appears in the last term of the sum on the right-hand side of $\eqref{thm}$. But surface tensors of lower rank appear in the remaining terms of the sum. Therefore, we need to express the lower rank tensors $\T$ for $k=0, \dots,  \frac{s}{2}-1$ as integrals. This can be done by using Theorem $\ref{crofton-like}$ with $s=2k$ for $k=0, \dots, \frac{s}{2}-1$. This way, we get $\frac{s}{2} + 1$ linear equations, which give rise to the linear system 
\begin{align*}
&\begin{pmatrix}
C_0\integral[0] \\
C_2\integral[2]\\
\vdots \\
 \\
C_s\integral
\end{pmatrix}
= C 
\begin{pmatrix}
\T[0] \\
\T[2] \\
\\
\vdots \\
 \\
\T[s]
\end{pmatrix}
\end{align*}
where
\begin{equation*}
C=\begin{pmatrix}
c_0^{(0)} & 0 & 0 & \dots & 0 \\
c_0^{(1)}Q & c_1^{(1)} & 0 & & \vdots \\
\vdots &&\ddots &  & 0\\
&&&&\\
c_0^{(\frac{s}{2})}Q^{\frac{s}{2}} & c_1^{(\frac{s}{2})}Q^{\frac{s}{2}-1} & \dots  &c_{\frac{s}{2}-1}^{(\frac{s}{2})}Q & c_{\frac{s}{2}}^{(\frac{s}{2})}
\end{pmatrix}
\end{equation*}
and $C_j=\frac{\pi j! \omega_{j+1}^2 \omega_n}{2\omega_{n+j+1}}$ for $j=0,2,4, \dots, s$.
Our aim is to express $\T[s]$ as an integral, hence we have to invert the system. Notice that the constants $c_i^{(i)}$ are non-zero, which ensures that the system actually is invertible. The system can be inverted by the matrix
\begin{equation}\label{invers matrix}
D=
\begin{pmatrix}
d_{00} & 0 & 0 & \dots && 0 \\
d_{10}Q & d_{11} & 0 & &&\vdots \\
d_{20}Q^2 & d_{21}Q & d_{22} & 0 \\
\vdots & & & \ddots & & 0 \\
d_{\frac{s}{2}0}Q^{\frac{s}{2}} & d_{\frac{s}{2}1}Q^{\frac{s}{2}-1} & \dots & & & d_{\frac{s}{2} \frac{s}{2}}
\end{pmatrix},
\end{equation}
where $d_{ii}=\frac{1}{c_i^{(i)}}$ for $i=0, \dots, \frac{s}{2}$, and $d_{ij}=-\frac{1} {c_i^{(i)}}\sum_{k=j}^{i-1}c_k^{(i)}d_{kj}$ for $i=1, \dots, \frac{s}{2}$ and $j=0, \dots, i-1$. In particular, we have
\begin{equation}\label{reffinal}
\T[s] = \sum_{j=0}^\frac{s}{2} d_{\frac{s}{2}j} Q^{\frac{s}{2}-j}C_{2j} \integral[2j].
\end{equation}
Notice that only the dimension of the matrix \eqref{invers matrix} depends on $s$, hence we get the same integral formulas for the lower rank tensors for different choices of $s$. Formula \eqref{Simpel form for T} and the above considerations give the following `inverse' version of the Crofton's formula.

\begin{theorem}\label{crofton2}
Let $K \in \K^n$ and let $s \in \N_0$ be even. Then
\begin{equation}\label{thm2}
\int_{\En_1}  G_s(\Eorigo)V_0(K \cap E)\, \mu_1^n (dE) = \T[s],
\end{equation}
where 
\begin{equation*}
G_{2m}(L):=\sum_{j=0}^m  \frac{2  d_{mj} C_{2j} }{(2j)! \, \omega_{2j+1}} Q^{m-j}Q(L)^j
\end{equation*} 
for $L \in \cLo$ and $m \in \N_0$.
\end{theorem}

It should be remarked that the measurement function in \eqref{thm2} is just a linear combination of the relative tensors of even rank at most $s$, but we prefer the present form to indicate the dependence on $K$ more explicitly.

\begin{example}\label{s er 4}\rm 
For $s=4$ the matrices are 
\begin{equation*}
C=
\begin{pmatrix}
2 & 0 & 0 \\
2Q & 8 \pi & 0 \\
2Q^2 & 16 \pi Q & -\frac{64 \pi^2}{3}
\end{pmatrix}
\end{equation*}
and
\begin{equation}\label{invers matrix ex}
D=
\begin{pmatrix}
\frac{1}{2} & 0 & 0 \\[1ex]
-\frac{1}{8 \pi}Q & \frac{1}{8 \pi} & 0 \\[1ex]
-\frac{3}{64 \pi^2}Q^2 & \frac{3}{32 \pi^2}Q & -\frac{3\pi^2}{64}
\end{pmatrix}.
\end{equation}
Since $C_0=\frac{2 \pi \omega_n}{\omega_{n+1}}$, $C_2=\frac{16 \pi^3 \omega_{n}}{\omega_{n+3}}$ and $C_4=\frac{256 \pi^5 \omega_n}{3 \omega_{n+5}}$, we have
\begin{equation*}
G_4(L)= -\frac{\omega_n}{32 \pi \omega_{n+1}}\big(3Q^2 - 6(n+1)QQ(L) + \pi^4(n+1)(n+3)Q(L)^2\big) ,
\end{equation*}
and
\begin{equation*}
G_2(L)=\frac{\omega_n}{4\omega_{n+1}} \bigg((n+1)Q(L)-Q \bigg) 
\end{equation*}
for $L \in \cLo.$
\end{example}

In Theorem \ref{crofton2} we only considered the situation, where $s$ is even. It is natural to ask whether $\T[s]$ can also be written as a linear Crofton integral when $s$ is odd. The case $s=1$ is trivial, as the tensor $\Phi_{n-1,0,1}(K)=0$ for all $K\in\K^n$.  If $n=1$, then $\T[s]=0$ for all odd $s$, since the area measure of order 0 is the Hausdorff measure on the sphere. Apart from these trivial examples, $\Phi_{n-1,0,s}$ cannot be written as a linear Crofton-type integral, when $s$ is odd and the measurement function satisfies some rather weak assumptions. This is shown in Theorem \ref{s ulige}.

\begin{theorem} \label{s ulige}
Let $n \geq 2$ and let $s>1$ be odd. Then there exists neither a translation invariant nor a bounded measurable measurement function $\nobreak{\alpha \colon \K^n \rightarrow \Tensor^s}$ such that
\begin{equation}\label{eq}
\int_{\En_1} \alpha(K \cap E) \, \mu_1^n(dE)=\T[s]
\end{equation}
for all $K \in \K^n.$
\end{theorem}
\begin{proof} Let $\alpha \colon \K^n \rightarrow \Tensor^s$ be a measurable and bounded function that satisfies equation \eqref{eq}. Since $\mu_1^n(\{E \in \En_1 \mid E \cap K = \emptyset\})=\infty$ for $K \in \K^n$, we have $\alpha(\emptyset)=0$. Now define the averaged function
\begin{equation*}
\alpha_r(M)=\frac{1}{V_n(rB^n)}\int_{rB^n}\alpha(M+x) \, \lambda (dx), \qquad M \in \K^n,
\end{equation*}
for $r > 0$. Since $\alpha$ is measurable and bounded, the average function $\alpha_r$ is well-defined. Clearly $\alpha_r(\emptyset)=0$. Using Fubini's theorem, the invariance of $\mu_1^n$ and the fact that $\Phi_{n-1,0,s}$ is translation invariant, we get that
\begin{equation*}
\int_{\En_1} \alpha_r(K \cap E) \, \mu_1^n (dE)
=\frac{1}{V_n(rB^n)} \int_{rB^n} \Phi_{n-1,0,s}(K + x) \, \lambda(dx)
=\T[s].
\end{equation*}
Let $K \in \K^n$ be such that $K \subseteq B^n$. Since $K \cap E$ is either the empty set or a a line segment in $B^n$ when $E \in \En_1$, there exists a vector $z_E \in \R^n$ with $\norm z_E \norm \leq 2$ such that $-(K \cap E)=(K \cap E) + z_E$. Let $\A =\{E \in \En_1 \mid B^n \cap E \neq \emptyset\}$,  
let $B_1\Delta B_2$ denote the symmetric difference of two sets $B_1,B_2$, and assume that $|\alpha|\le M$ for some constant $M$.  Then
\begin{align*}
& \big|\T[s] - \Phi_{n-1,0,s}(-K)\big|
=\bigg| \int_{\A} \alpha_r(K \cap E) - \alpha_r(-(K \cap E)) \, \mu_1^n (dE) \bigg|
\\
&\qquad \leq \frac{1}{V_n(rB^n)}   \int_\A \bigg| \int_{rB^n} \alpha((K \cap E) + x)\, \lambda(dx) \\
&\qquad\qquad\qquad\qquad\qquad\qquad\qquad -  \int_{rB^n+z_E} \alpha((K \cap E) + x)\, \lambda (dx)  \bigg|  \,\mu_1^n (dE)\\
&\qquad \leq \frac{1}{V_n(rB^n)} \int_{\A}\int_{(rB^n+z_E)\Delta (rB^n)} \left|\alpha((K\cap E)+x)\right|\, \lambda(dx)\, \mu_1^n (dE)\\
&\qquad\leq \frac{2M}{V_n(rB^n)} \int_{\A} V_n((rB^n+z_E)\setminus (rB^n))\, \mu_1^n (dE)\\
&\qquad\leq 2M\frac{ (r+2)^n-r^n}{r^n}\kappa_{n-1}\longrightarrow 0\qquad\text{as $r \rightarrow \infty$. }
\end{align*}
Here we used that $(rB^n+z_E)\setminus (rB^n)\subseteq (r+2)B^n\setminus (rB^n)$ and $\mu^n_1(\A)=\kappa_{n-1}$.

Hence, we get $\T[s]=\Phi_{n-1,0,s}(-K)$. Since $s$ is odd, we also have  $\T[s]=-\Phi_{n-1,0,s}(-K)$. Therefore $\T[s]=0$, which is not the case for all $K \subseteq  B^n$, since $s>1$. Then, by contradiction, \eqref{eq} cannot be satisfied by a bounded measurement function, when $s>1$ is odd.

Now assume that $\alpha$ is translation invariant and satisfies equation \eqref{eq}. As $-(K \cap E)$ is a translation of $K \cap E$, we have
\begin{equation*}
\int_{\En_1} \alpha(-K \cap E) \, \mu_1^n (dE)
= \int_{\En_1} \alpha(-(K \cap E)) \, \mu_1^n (dE)
=\int_{\En_1} \alpha(K \cap E) \, \mu_1^n (dE),
\end{equation*}
implying $\Phi_{n-1,0,s}(-K)=\T[s]=-\Phi_{n-1,0,s}(-K)$, and hereby we obtain that  $\T[s]=0$ for all $K \in \K^n$. This is a contradiction as before.
\end{proof}

\section{Design based estimation}\label{sec estimation}
In this section we use the integral formula \eqref{thm2} in Theorem \ref{crofton2} to derive unbiased estimators of the surface tensors $\T[s]$ of $K \in \K^n$, when $s$ is even. We assume throughout this chapter that $n \geq 2$. Three different types of estimators based on 1-dimensional linear sections are presented. First, we establish estimators based on isotropic uniform random lines, then estimators based on random lines in vertical sections and finally estimators based on non-isotropic uniform random lines.

\subsection{Estimation based on isotropic uniform random lines} \label{IUR}
In this section we construct estimators of $\T[s]$ based on isotropic uniform random lines. Let $K \in \K^n.$ We assume that (the unknown set) $K$ is contained in a compact reference set $A \subseteq \R^n$, the latter being known. Now let $E$ be an \textit{isotropic uniform random (IUR) line in $\R^n$ hitting $A$}, i.e., the distribution of $E$ is given by
\begin{equation} \label{IURdistribution}
\Prob(E \in \A)= c_1(A) \int_{\A} \1(E' \cap A \neq \emptyset)\,\mu_1^n(dE')
\end{equation}
for $\A \in \B(\En_1),$ where $c_1(A)$ is the normalizing constant 
\begin{equation*}
c_1(A)=\bigg(\int_{\En_1} \1(E' \cap A \neq \emptyset)\,\mu_1^n(dE')\bigg)^{-1}.
\end{equation*}
By \eqref{thm} with $ s=0 $ the normalizing constant becomes $c_1(A)=\frac{\omega_n }{2 \kappa_{n-1}}V_{n-1}(A)^{-1}$, when $A$ is a convex body.
Then Theorem \ref{crofton2} implies that
\begin{equation}\label{Estimator}
c_1(A)^{-1} G_{s}(\Eorigo)V_0(K \cap E)
\end{equation}
is an unbiased estimator of $\T[s]$, when $s$ is even.

\begin{example} \label{example}\rm 
Using the expressions of $G_2$ and $G_4$ in Example $\ref{s er 4}$ we get that  
\begin{align*}
-\frac{V_{n-1}(A)}{32\pi^2}\big(3Q^2 - 6(n+1)QQ(L) + \pi^4(n+1)(n+3)Q(L)^2\big) V_0(K \cap E)
\end{align*}
is an unbiased estimator of $\T[4],$ and
\begin{align}\label{s er 2}
\frac{V_{n-1}(A)}{4\pi}\bigg((n+1)Q(\Eorigo)-Q \bigg) V_0(K \cap E)
\end{align}
is an unbiased estimator of $\T[2]$, when $A$ is a convex body. For $n=3$, these estimators read
\begin{equation*}
-\frac{V_2(A)}{32 \pi^2}\bigg(3 Q^2-24QQ(\Eorigo)+24\pi^4Q(\Eorigo)^2\bigg)V_0(K \cap E)
\end{equation*}
and
\begin{equation}\label{IURn3}
\frac{V_2(A)}{\pi} \bigg(Q(\Eorigo)-\frac{1}{4}Q\bigg)V_0(K \cap E).
\end{equation}
\end{example} 
 
An investigation of the estimators in Example \ref{example} shows that they possess some unfavourable statistical properties. If $K \cap E = \emptyset$ the estimators are simply zero. Furthermore, if $K \cap E \neq \emptyset,$ the matrix representation of the estimator \eqref{s er 2} of $\T[2]$ is, in contrast to $ \T[2] $, not positive semi-definite. In fact, the eigenvalues of the matrix representation of $(n+1)Q(\Eorigo)-Q$ are $n$ (with multiplicity $ 1 $) and $-1$ (with multiplicity $n-1$). It is not surprising that estimators based on the measurement of one single line, are not sufficient, when we are estimating tensors with many unknown parameters. To improve the estimators, they can be extended in a natural way to use information from $N$ \textit{IUR} lines for some $N \in \N$. In addition,  the integral formula \eqref{thm2} can be rewritten in the form
\begin{align} \label{Alternativ}
\T[s]&=\int_{\cLo} \int_{L^\perp} G_s(L)V_0(K \cap (L+x))\, \lambda_{\Lp}(dx) \, \nu_1^n (dL) \nonumber \\
&=\int_{\cLo} G_s(L) V_{n-1}(K \vert L^\perp) \, \nu_1^n (dL),
\end{align}
which implies that
\begin{equation}\label{projestimator}
\frac{1}{N}\sum_{i=1}^NG_s(L_i)V_{n-1}(K \vert L_i^\perp)
\end{equation}
is an unbiased estimator of $\T[s]$, when $L_1, \dots L_N \in \cLo$ are $ N $ isotropic lines (through the origin) for an $ {N \in \N} $. When $K$ is full-dimensional this estimator never vanishes. In the case where $s=2$ the estimator becomes
\begin{equation}\label{EstimatorN}
\frac{1}{N}\frac{\omega_n}{4\omega_{n+1}}\sum_{i=1}^N\big((n+1)Q(L_i)-Q \big) V_{n-1}(K \vert L_i^\perp).
\end{equation}
In stereology it is common practice to use orthogonal test lines.  If we set $ N=n $ and let $ L_1, \dots, L_n $ be isotropic, pairwise orthogonal lines, then the estimator \eqref{EstimatorN} becomes positive definite exactly when 
\begin{equation}\label{posdefcondition}
(n+1)V_{n-1}(K \vert L_i^\perp) > \sum_{j=1}^n V_{n-1}(K \mid L_j^\perp)
\end{equation}
for all $i=1, \dots, n.$ This is a condition on $ K $ requiring that $ K $ is not too eccentric. A sufficient condition for \eqref{posdefcondition} to hold makes use of the radius $ R(K) $ of the smallest ball containing $ K $ and the radius $ r(K) $ of the largest ball contained in $ K $. If
\begin{equation}
\frac{r(K)}{R(K)}> \bigg(1-\frac{1}{n}\bigg)^{\frac{1}{n-1}},
\end{equation}
then \eqref{posdefcondition} is satisfied, and hence the estimator \eqref{EstimatorN} with $ n $ orthogonal, isotropic lines is positive definite. In $ \R^2 $ this means that $ 2r(K) > R(K) $ is sufficient for a positive definite estimator \eqref{EstimatorN}, and in particular for all ellipses for which the length of the longer main axis  does not exceed twice the length of the smaller main axis, \eqref{EstimatorN} yields positive definite estimators. For ellipses, this criterion is also necessary as the following example shows. 
\begin{example}\rm
Consider the situation where $ n=2 $ and $ K $ is an ellipse, $ K=\{x \in \R^2 \mid x^\top B x \leq 1\} $, given by the matrix 
\begin{equation*}
 B=\begin{pmatrix}
\alpha^{-2} & 0 \\ 0 & (k\alpha)^{-2}
\end{pmatrix},
\end{equation*}
where $ \alpha > 0 $ and $ k \in (0,1] $. The parameter $ k $ determines the eccentricity of $ K $. If $ k \in (\frac{1}{2},1] $, and $ L_1 $ and $ L_2 $ are orthogonal, isotropic random lines in $ \R^2 $, the estimator \eqref{EstimatorN} becomes positive definite by the above considerations. Now let $ k \in [0,1/2] $. Since $ n=2 $, each pair of orthogonal lines is determined by a constant $ \phi \in [0, \frac{\pi}{2}) $ by letting $ L_1 = u_{\phi}^\perp $ and $ L_2=u_{\phi+ \frac{\pi}{2}}^\perp $, where $ u_{\phi}=(\cos(\phi),\sin(\phi))^\top $. Then
\begin{equation*}
V_{n-1}(K \mid L_1^\perp)=2h(K,u_{\phi})=2 \alpha \sqrt{\cos^2(\phi)+k^2 \sin^2(\phi)}
\end{equation*}
and
\begin{equation*}
V_{n-1}(K \mid L_2^\perp ) = 2 \alpha \sqrt{\sin^2(\phi)+k^2 \cos^2(\phi)}.
\end{equation*}
Condition \eqref{posdefcondition} is satisfied if and only if
\begin{equation*}
\phi \in [\sin^{-1}\bigg(\sqrt{\frac{1-4k^2}{5(1-k^2)}}\bigg), \cos^{-1}\bigg(\sqrt{\frac{1-4k^2}{5(1-k^2)}}\bigg) ],
\end{equation*}
and the probability that the estimator is positive definite, when $ L_1 $ and $ L_2 $ are orthogonal, isotropic lines (corresponding to $ \phi $ being uniformly distributed on $ [0,\frac{\pi}{2}] $) is
\begin{equation*}
\frac{2}{\pi}\bigg(\cos^{-1}\bigg(\sqrt{\frac{1-4k^2}{5(1-k^2)}}\bigg)-\sin^{-1}\bigg(\sqrt{\frac{1-4k^2}{5(1-k^2)}}\bigg)\bigg),
\end{equation*}
which converges to $ \frac{2}{\pi} \big(\cos^{-1}(\sqrt{\frac{1}{5}})- \sin^{-1}(\sqrt{\frac{1}{5}}) \big) \approx 0.41 $ as $ k $ converges to 0. 
\end{example}

In $ \R^2 $ the estimator \eqref{EstimatorN} can alternatively be combined with a systematic sampling approach with $ N $ isotropic random lines. Let $ N \in \N $, and let $ \phi_0 $ be uniformly distributed on $ [0,\frac{\pi}{N}] $. Moreover, let $ \phi_i=\phi_0 + i \frac{\pi}{N} $ for $ i=1, \dots, N-1 $. Then $ u_{\phi_0}, \dots, u_{\phi_{N-1}} $ are $ N $ systematic isotropic uniform random directions in the upper half of $ S^1 $, where $ u_{\phi}=(\cos(\phi),\sin(\phi))^\top $. As the estimator \eqref{EstimatorN} is a tensor of rank 2, it can be identified with the symmetric $2 \times 2$ matrix, where the $(i,j)$'th entry is the estimator evaluated at $(e_i,e_j)$, where $(e_1,e_2)$ is the standard basis of $\R^2$. The estimator becomes
\begin{equation}\label{syst}
S_N(K, \phi_0)=\frac{1}{N}\sum_{i=0}^{N-1} \begin{pmatrix}
3\cos^2(\phi_i)-1 & 3\cos(\phi_i) \sin(\phi_i) \\
3\cos(\phi_i) \sin(\phi_i) & 3\sin^2(\phi_i)-1
\end{pmatrix} V_1(K \mid u_{{\phi}_i}^\perp).
\end{equation}

\begin{example} \label{exsyst}\rm 
To investigate how the estimator $ S_N(K, \phi_0) $ performs we estimate the probability that the estimator is positive definite for three different origin-symmetric convex bodies in $ \R^2 $;  a parallelogram, a rectangle, and an ellipse. Thus let
\begin{align*} 
K_1&=\mathrm{conv}\{ (1,\epsilon),(-1,\epsilon),(-1,-\epsilon),(1,-\epsilon)\}, \\
K_2&=\mathrm{conv}\{ (1,0),(0,\epsilon),(-1,0),(0,-\epsilon)\} \\
\intertext{and}
K_3&= \{x \in \R^2 \mid x^\top \begin{pmatrix}
1  & 0 \\ 0 & \frac{1}{\sqrt{\epsilon}}
\end{pmatrix} x \leq 1\}  
\end{align*}
with $\epsilon=0.1$.
The support functions, and hence the intrinsic volumes $V_1(K_i \vert u_\phi^\perp)$, of $ K_1, K_2$ and $ K_3 $ have simple analytic expressions, and the estimator $ S_N(K_i, \phi_0) $ can be calculated for $ \phi_0 \in [0,\frac{\pi}{N}] $ and $ i=1,2,3 $. The eigenvalues of the estimators can be calculated numerically, and the probability that the estimators $ S_N(K_i,\phi_0) $ are positive definite, when $ \phi_0 $ is uniformly distributed on $ [0, \frac{\pi}{N}] $, can hereby be estimated. For each choice of $ N $, the estimate of the probability is based on $ 500 $ equally spread values of $ \phi_0 $ in $ [0,\frac{\pi}{N}] $. The estimate of the probability that $ S_N(K_i,\phi_0) $ is positive definite is plotted against the number of equidistant lines $ N $ for $i=1,2,3$ in Figure \ref{Plot of prob for posdef}. The plots in Figure \ref{Plot of prob for posdef} show that even though we consider rather eccentric shapes, the number $N$ of lines needed to get a positive definite estimator with probability $1$ is in all cases less than $7$.
\end{example}

\begin{figure}
\centering
\begin{minipage}[c]{0.47 \textwidth}
\centering
\includegraphics[scale=0.26]{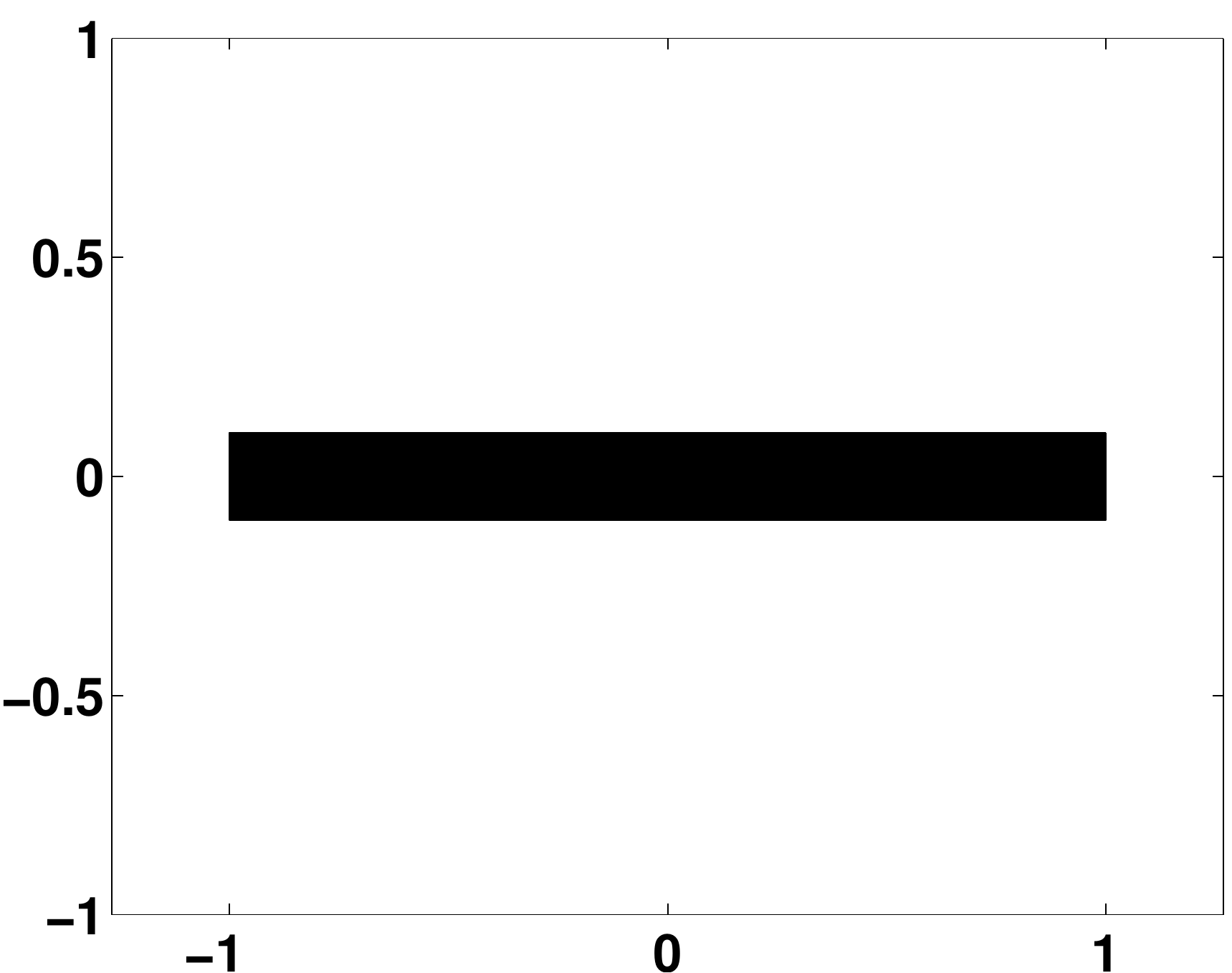}
\end{minipage}
\begin{minipage}[c]{0.47 \textwidth}
\centering
\includegraphics[scale=0.35]{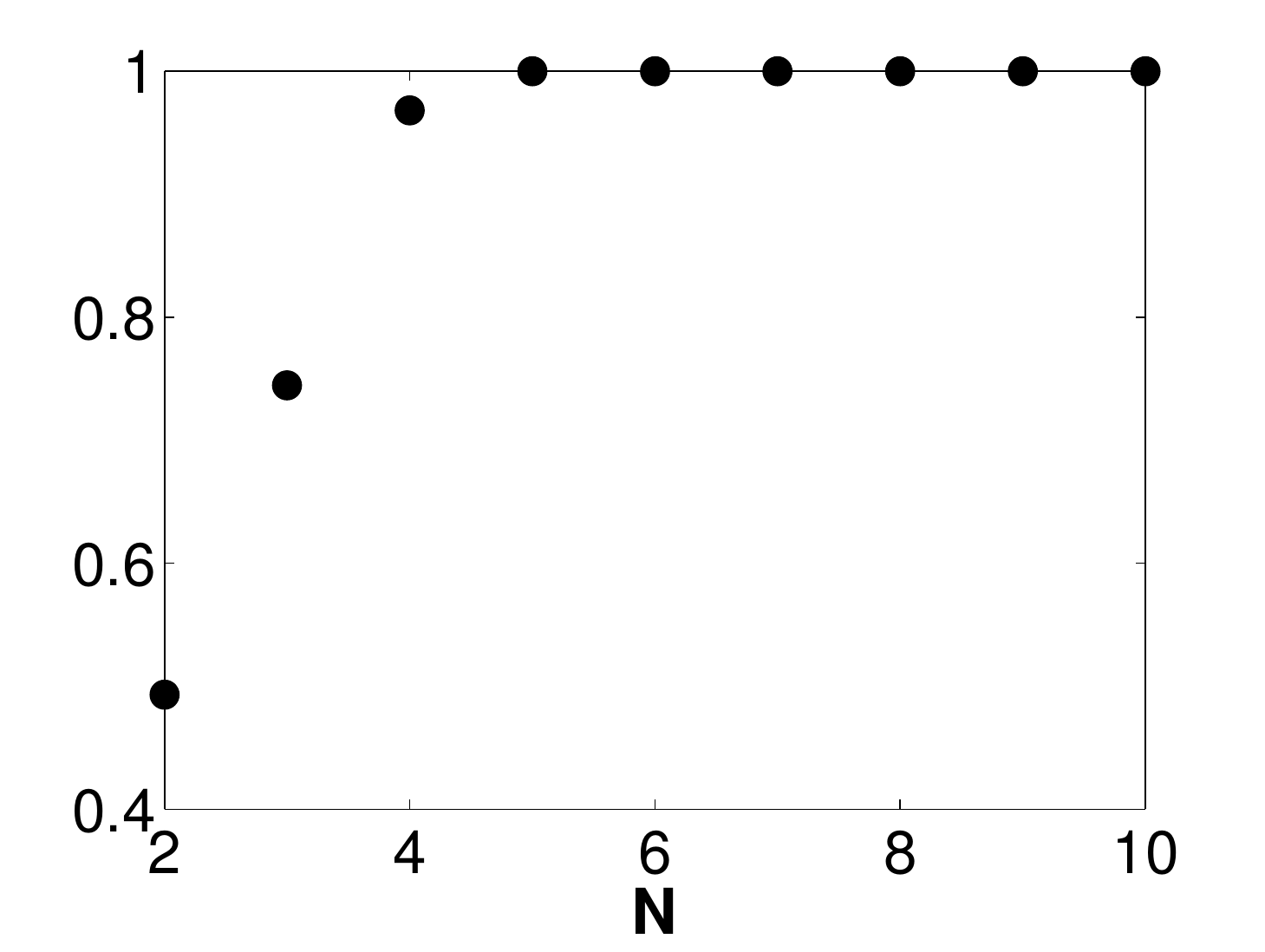}
\end{minipage}
\\
\begin{minipage}[c]{0.47 \textwidth}
\centering
\includegraphics[scale=0.26]{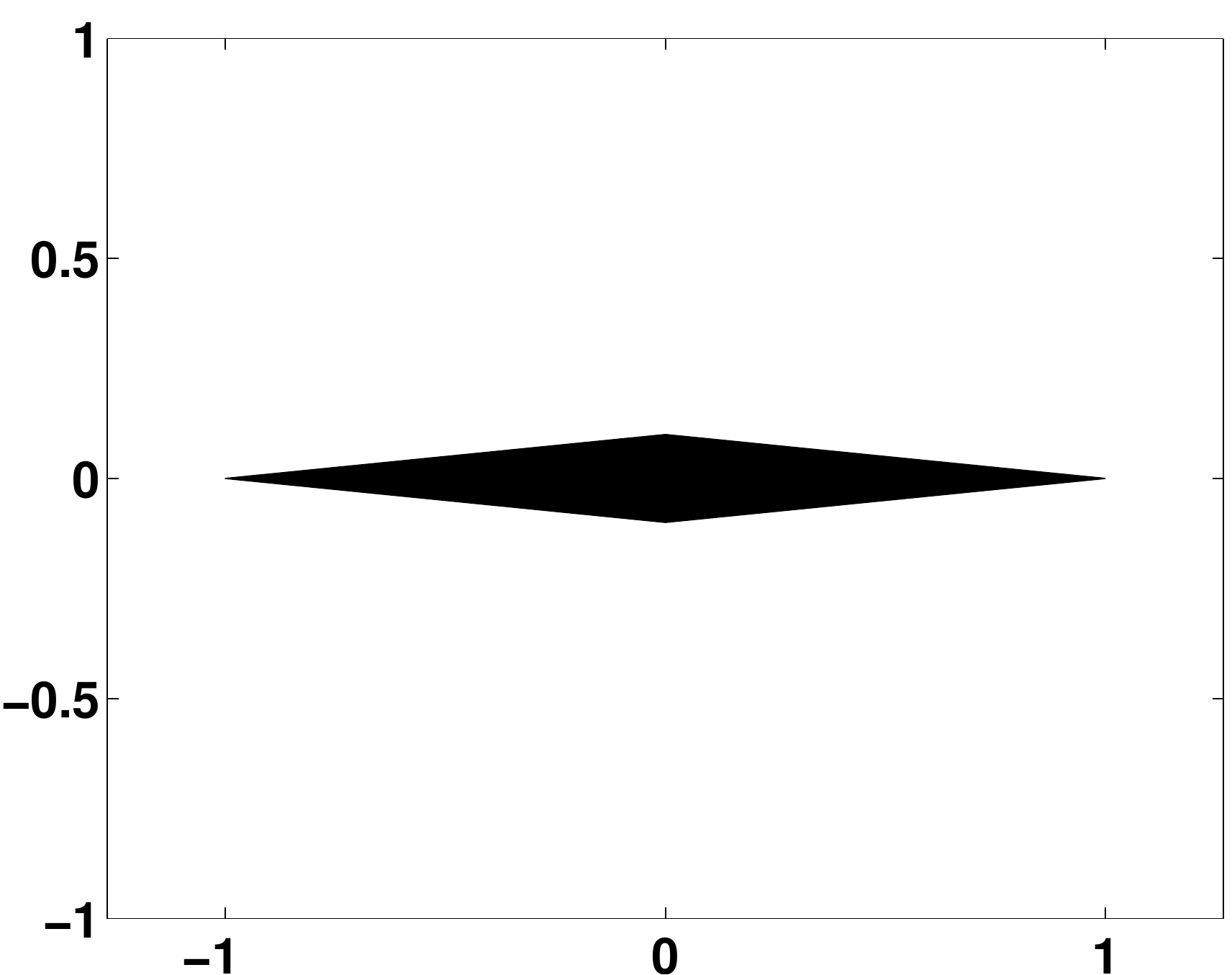}
\end{minipage}
\begin{minipage}[c]{0.47 \textwidth}
\centering
\includegraphics[scale=0.35]{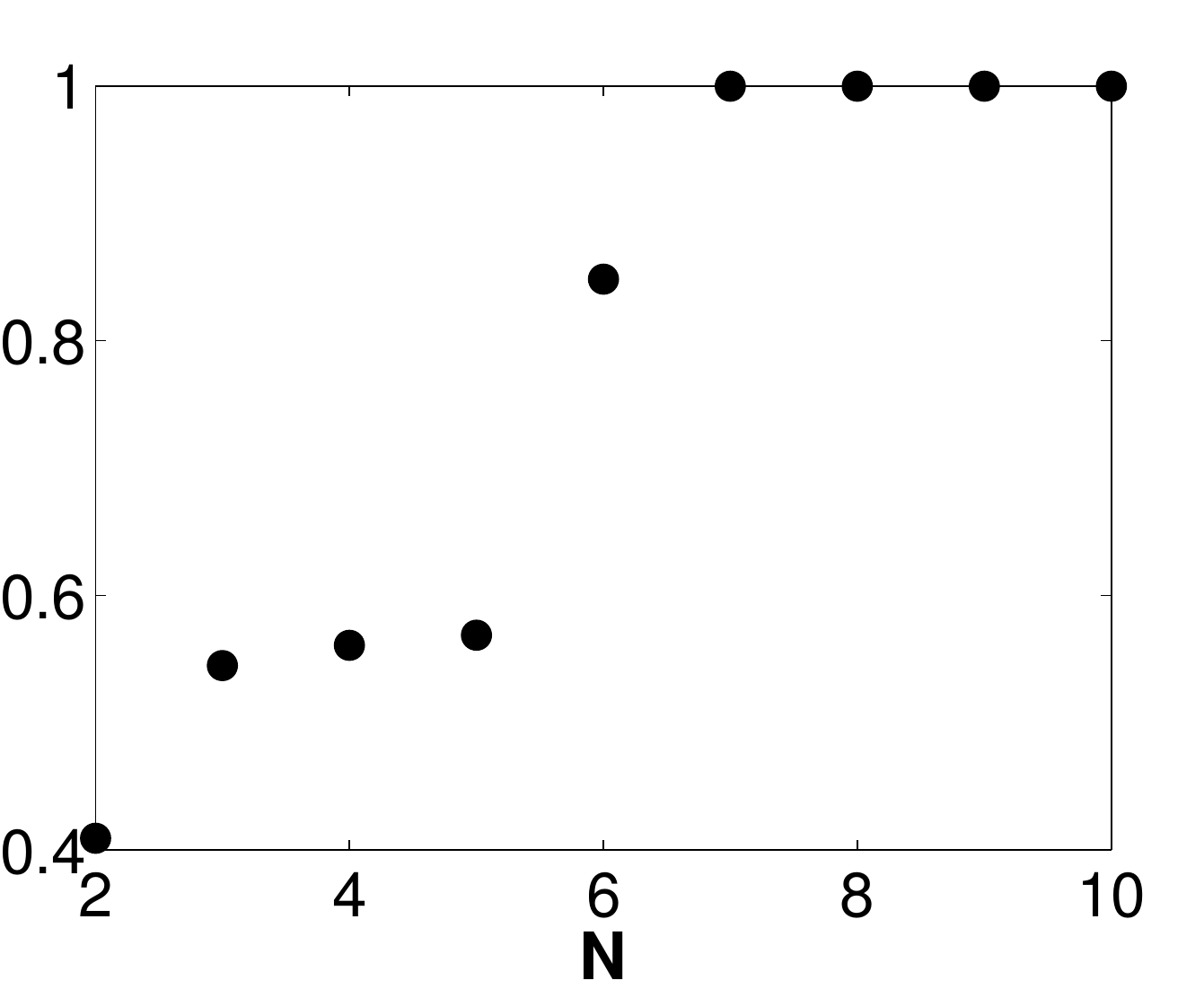}
\end{minipage}
\\
\begin{minipage}[c]{0.47 \textwidth}
\centering
\includegraphics[scale=0.26]{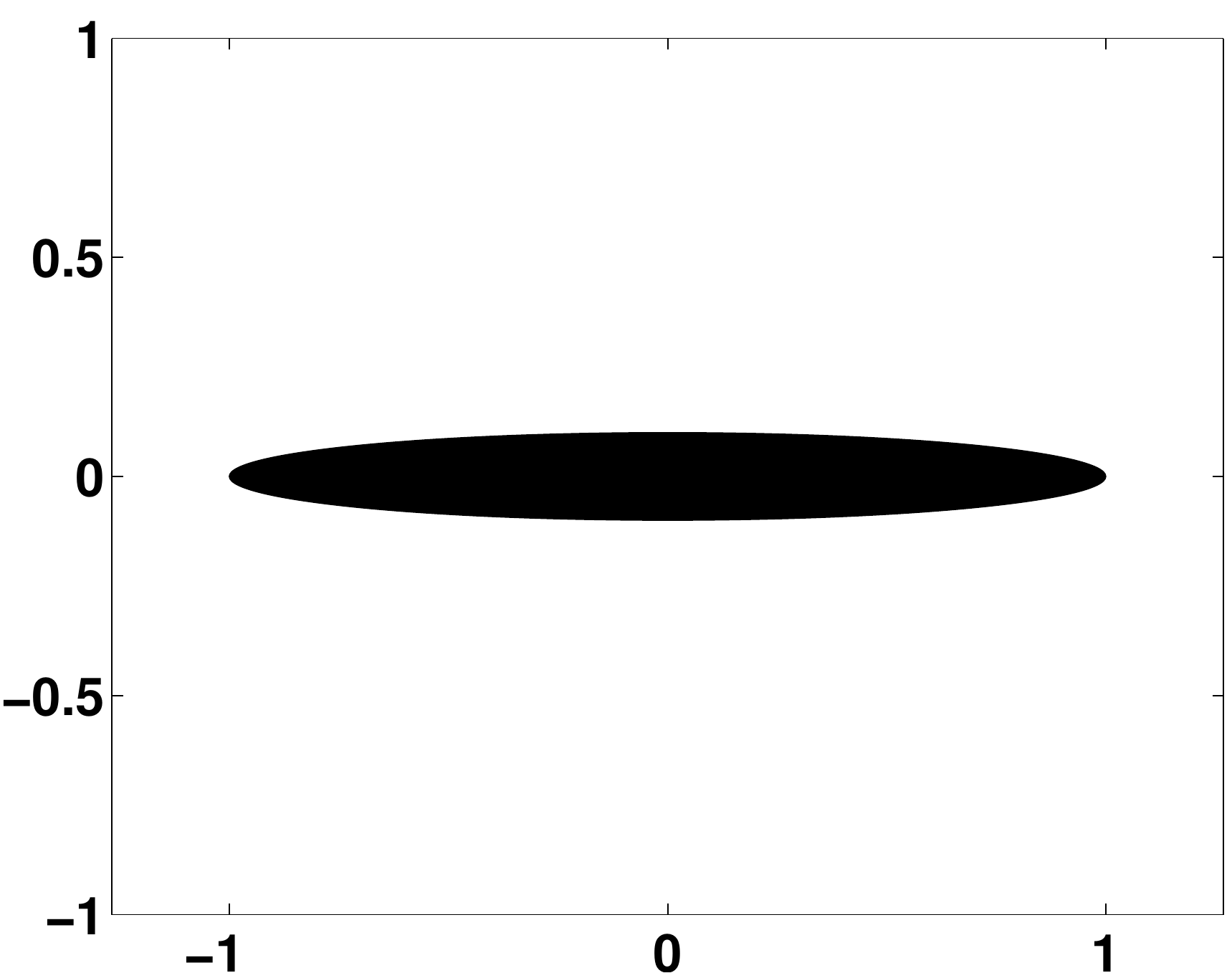}
\end{minipage}
\begin{minipage}[c]{0.47 \textwidth}
\centering
\includegraphics[scale=0.35]{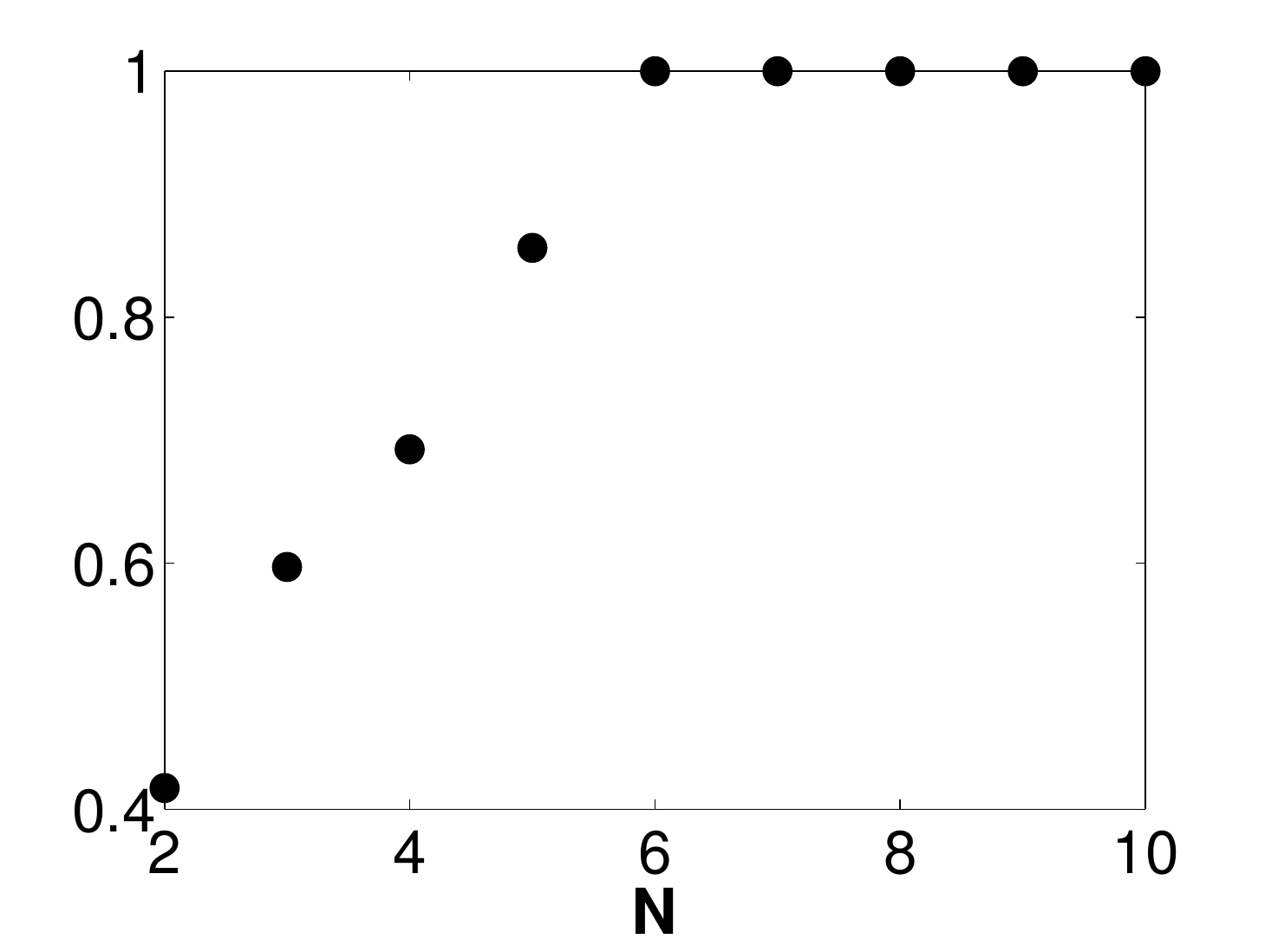}
\end{minipage}
\caption{The probability that $ S_N(K_i,\phi_0) $ is positive definite for $ i=1,2,3 $, when $ \phi_0 $ is uniformly distributed on $ [0,\frac{\pi}{N}] $ plotted against the number of equidistant lines $ N $.}
\label{Plot of prob for posdef}
\end{figure}  

To apply the estimator \eqref{Estimator} it is only required to observe whether the test line hits or misses the convex body $K$. The estimator \eqref{projestimator} requires more sophisticated information in terms of the projection function. In the following example the coefficient of variation of versions of the estimators \eqref{IURn3} and \eqref{EstimatorN} are estimated and compared in a three-dimensional set-up.  

\begin{example} \label{3dimex}\rm 
Let $K_{l}'$ be the prolate spheroid in $\R^3$ with main axis parallel to the standard basis vectors $e_1,e_2$ and $e_3$, and corresponding lengths of semi-axes $\lambda_1=\lambda_2=1$ and $\lambda_3=l$.
For $l=1, \dots, 5$, let $K_l$ denote the ellipsoid obtained by rotating $K_l'$ first around $e_1$ with an angle $\frac{3\pi}{16}$, and then around $e_2$ with an angle $\frac{5\pi}{16}$. Note, that the eccentricity of $K_l$ increases with $l$. In this example, based on simulations, we estimate and compare the coefficient of variation \textit{(CV)} of the developed estimators of $\Phi_{2,0,2}(K_l)$ for $l=1, \dots, 5$.

Formula \eqref{IURn3} provides an unbiased estimator of the tensor $\Phi_{2,0,2}(K_l)$ for $l=1, \dots, 5$. The estimator is based on one \textit{IUR} line hitting a reference set $A$, and can in a natural way be extended to an estimator based on three orthogonal \textit{IUR} lines hitting $A$. We estimate the variance of both estimators. Let, for $l=1, \dots,5$, the reference set $A_l$ be a ball of radius $R_l>0$. The choice of the reference set influences the variance of the estimator. In order to minimize this effect in the comparison of the CV's, the radii of the reference sets are chosen such that the probability that a test line hits $K_l$ is constant for $l=1, \dots, 5$. By formula \eqref{IURdistribution} the probability that an \textit{IUR} line hitting $A_l$ hits $K_l$ is $\frac{V_{2}(K_l)}{V_2(A_l)}$. The radius is chosen, such that this probability is $\frac{1}{7}$.
We further estimate the variance of the projection estimator \eqref{EstimatorN} based on one isotropic line and on three orthogonal isotropic lines.

As $\Phi_{2,0,2}(K_l)$ is a tensor of rank 2, it can be identified with the symmetric $3 \times 3$ matrix $\{\Phi_{2,0,2}(K_l)(e_i,e_j)\}_{i,j=1}^3$. Thus, in order to estimate $\Phi_{2,0,2}(K_l)$, the matrix $\{\hat{\Phi}_{2,0,2}(K_l)(e_i,e_j)\}_{i,j=1}^3$ is calculated. Here, $\hat{\Phi}_{2,0,2}(K_l)$ refers to any of the four estimators described above. Due to symmetry, there are six different components of the matrices. 

The estimates of the variances are based on 1500-10000 estimates of the tensor, depending on the choice of the estimator and the eccentricity of $K_l$. Using the estimates of the variances, we estimate the absolute value of the CV's by
\begin{equation*}
\widehat{CV}_{ij}=\frac{\sqrt{\widehat{\text{Var}}(\hat{\Phi}_{2,0,2}(K_l)(e_i,e_j))}}{|\Phi_{2,0,2}(K_l)(e_i,e_j)|},
\end{equation*}
for $i,j=1,2,3$ and $l=1, \dots, 5$. As $K_l$ is an ellipsoid, the tensor $\Phi_{2,0,2}(K_l)$ can be calculated numerically. The CV's of the four estimators are plotted in Figure \ref{4variances} for each of the six different components of the associated matrix. As $K_1$ is a ball, the off-diagonal elements of the matrix associated with $\Phi_{2,0,2}(K_1)$ are zero. Thus, the CV is in this case calculated only for the estimators of the diagonal-elements.

The projection estimators give, as expected, smaller CV's, than the estimators based on the Euler characteristic of the intersection between the test lines and the ellipsoid. For the estimators based on one test line the CV of the projection estimator is typically around $38\%$ of the corresponding estimator \eqref{IURn3}. For the estimators based on three orthogonal test lines, the CV of the projection estimator is typically $9\%$ of the estimator \eqref{IURn3}, when $l=2,\dots,5$. Due to the fact that $K_1$ is a ball, the variance of the projection estimator based on three orthogonal lines is 0, when $l=1$.

It is interesting to compare the increase of efficiency when using the estimator based on three orthogonal test lines instead of three i.i.d. test lines. The CV of an estimator based on three i.i.d. test lines is $\frac{1}{\sqrt{3}}$ of the CV of the estimator \eqref{IURn3}, (the ``$+$'' signs in Figure \ref{4variances}). The CV, when using three orthogonal test lines, is typically around $92\%$ of that CV. For $l=2,\dots, 5$, the CV's of the projection estimator based on three orthogonal lines, are typically $20\%$ of the CV, when using three i.i.d lines, indicating that spatial random systematic sampling increases precision without extra workload.

The CV's of the estimators of the diagonal-elements $\Phi_{2,0,2}(K_l)(e_i,e_i)$ are almost constant in $l$. Hence the eccentricity of $K_l$ does not affect the CV's for these choices of $l$. There is a decreasing tendency of the CV's of the estimators of the off-diagonal elements. This might be explained by the fact that the true value of $\Phi_{2,0,2}(K_l)(e_i,e_j)$ is close to zero, when $i \neq j$ and $l$ is small.
\end{example}

\begin{figure}
\begin{minipage}[c]{0.47 \textwidth}
\centering
\includegraphics[scale=0.30]{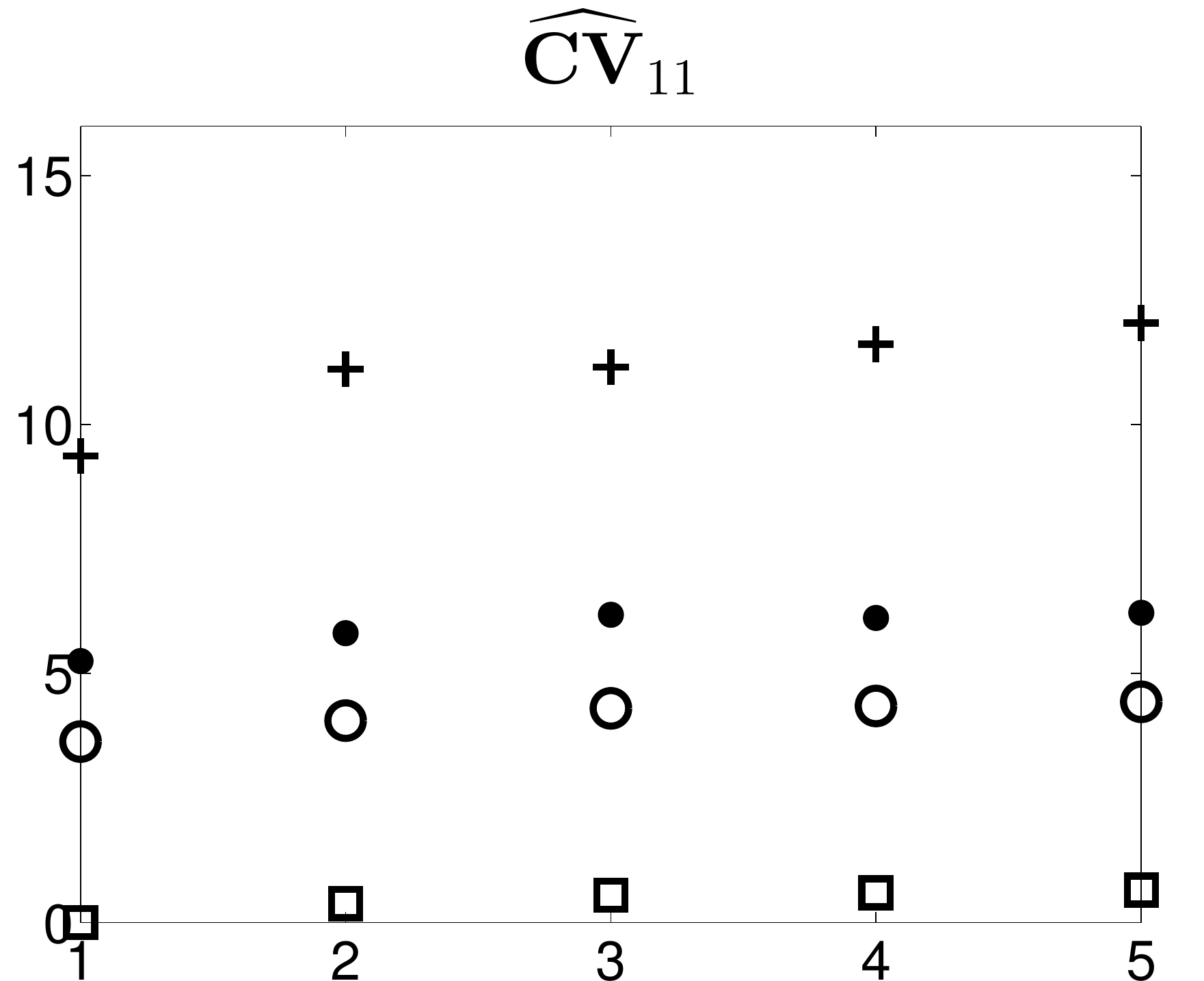}
\end{minipage}
\hfill
\begin{minipage}[c]{0.47 \textwidth}
\centering
\includegraphics[scale=0.30]{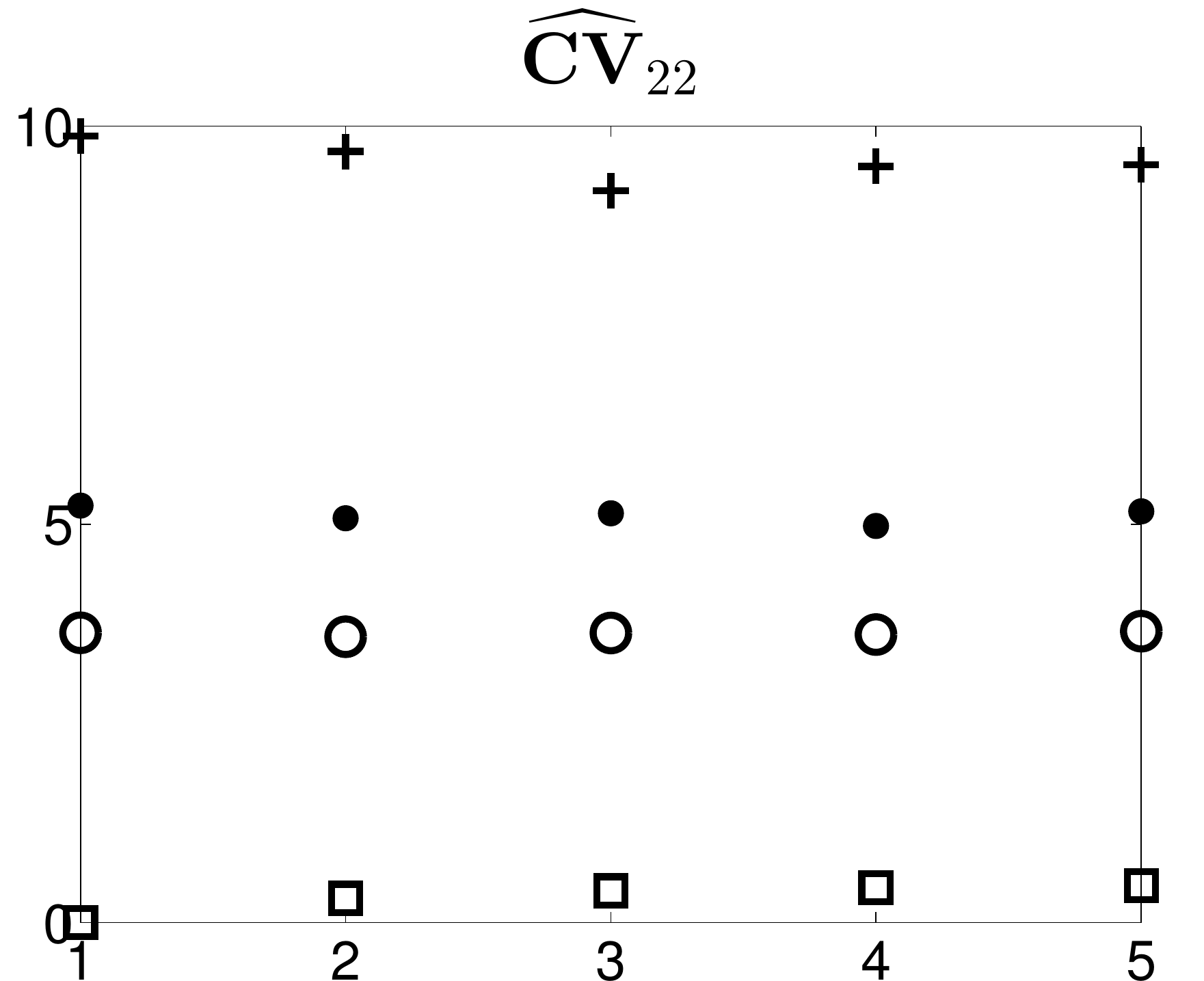}
\end{minipage}
\vspace{4mm}
\\
\begin{minipage}[c]{0.47 \textwidth}
\centering
\includegraphics[scale=0.30]{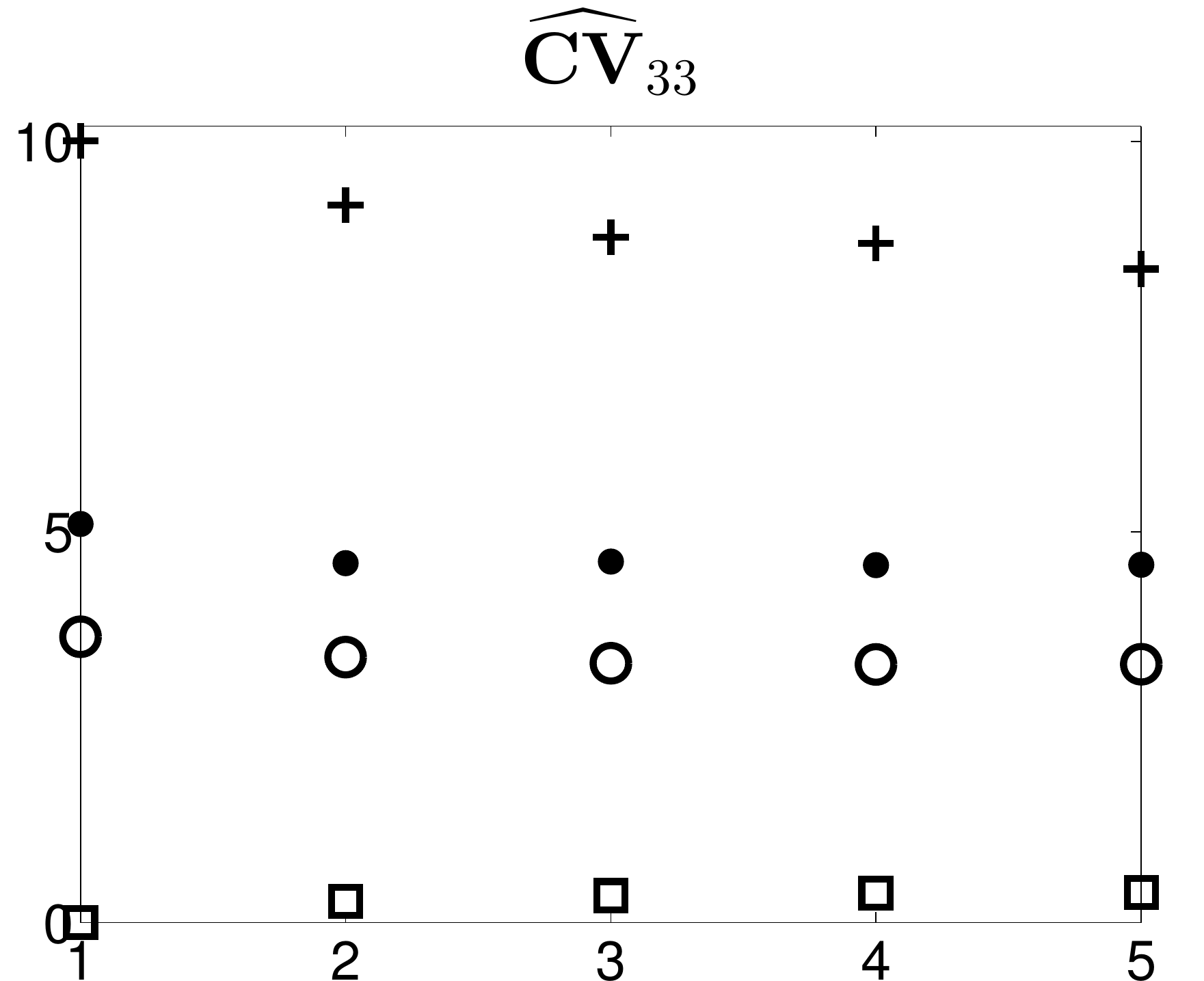}
\end{minipage}
\hfill
\begin{minipage}[c]{0.47 \textwidth}
\centering
\includegraphics[scale=0.30]{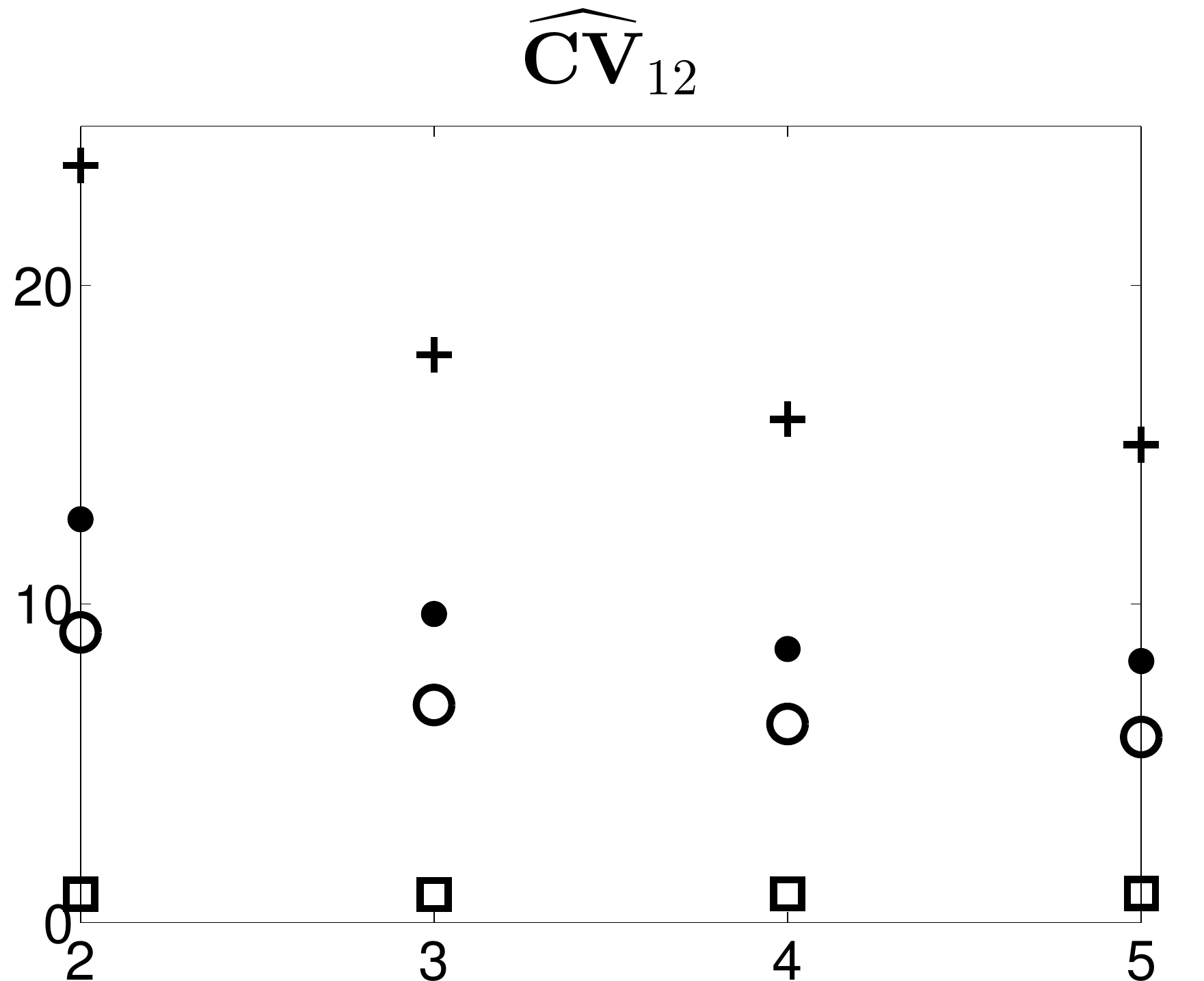}
\end{minipage}
\vspace{4mm}
\\
\begin{minipage}[c]{0.47 \textwidth}
\centering
\includegraphics[scale=0.30]{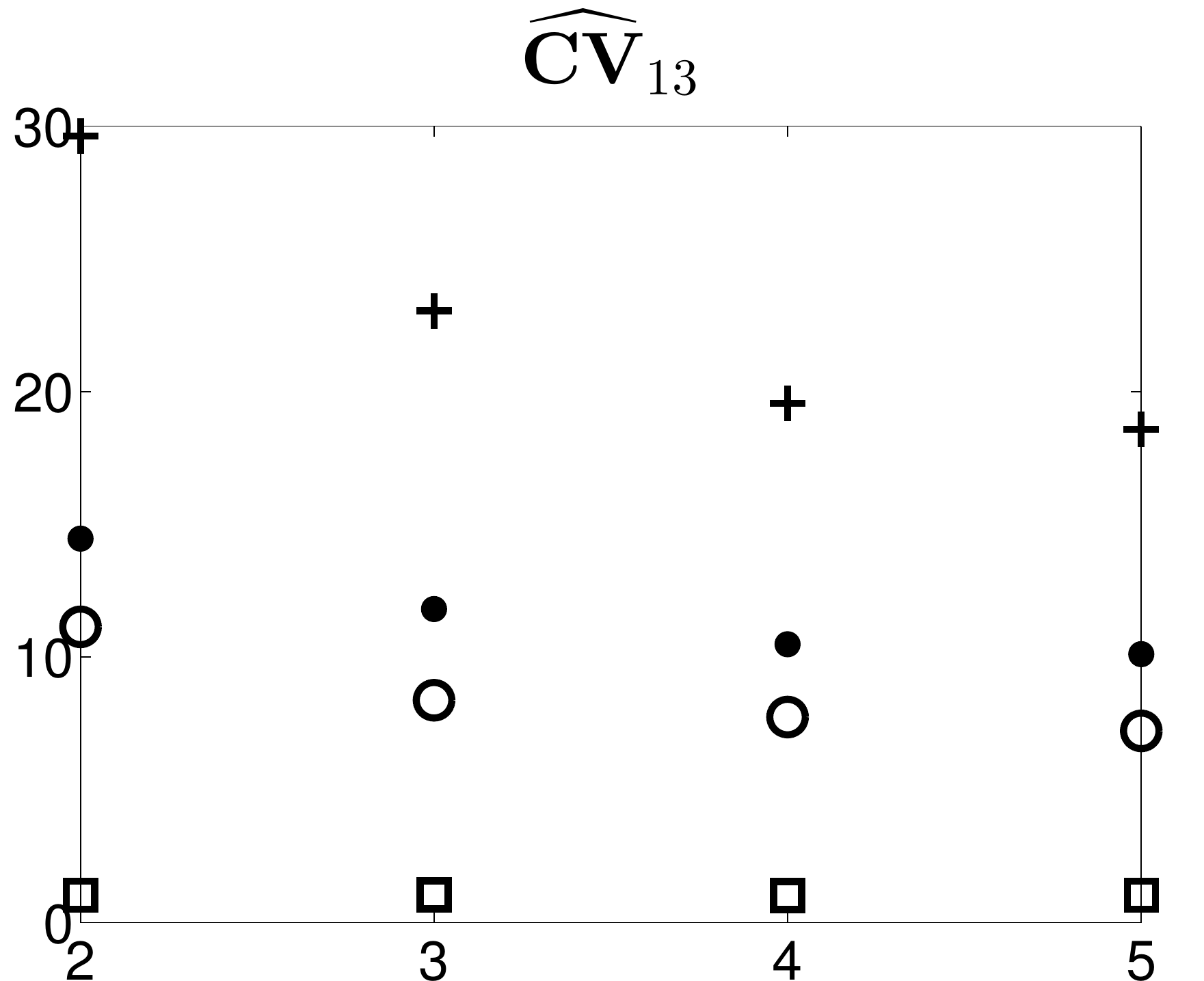}
\end{minipage}
\hfill
\begin{minipage}[c]{0.47 \textwidth}
\centering
\includegraphics[scale=0.30]{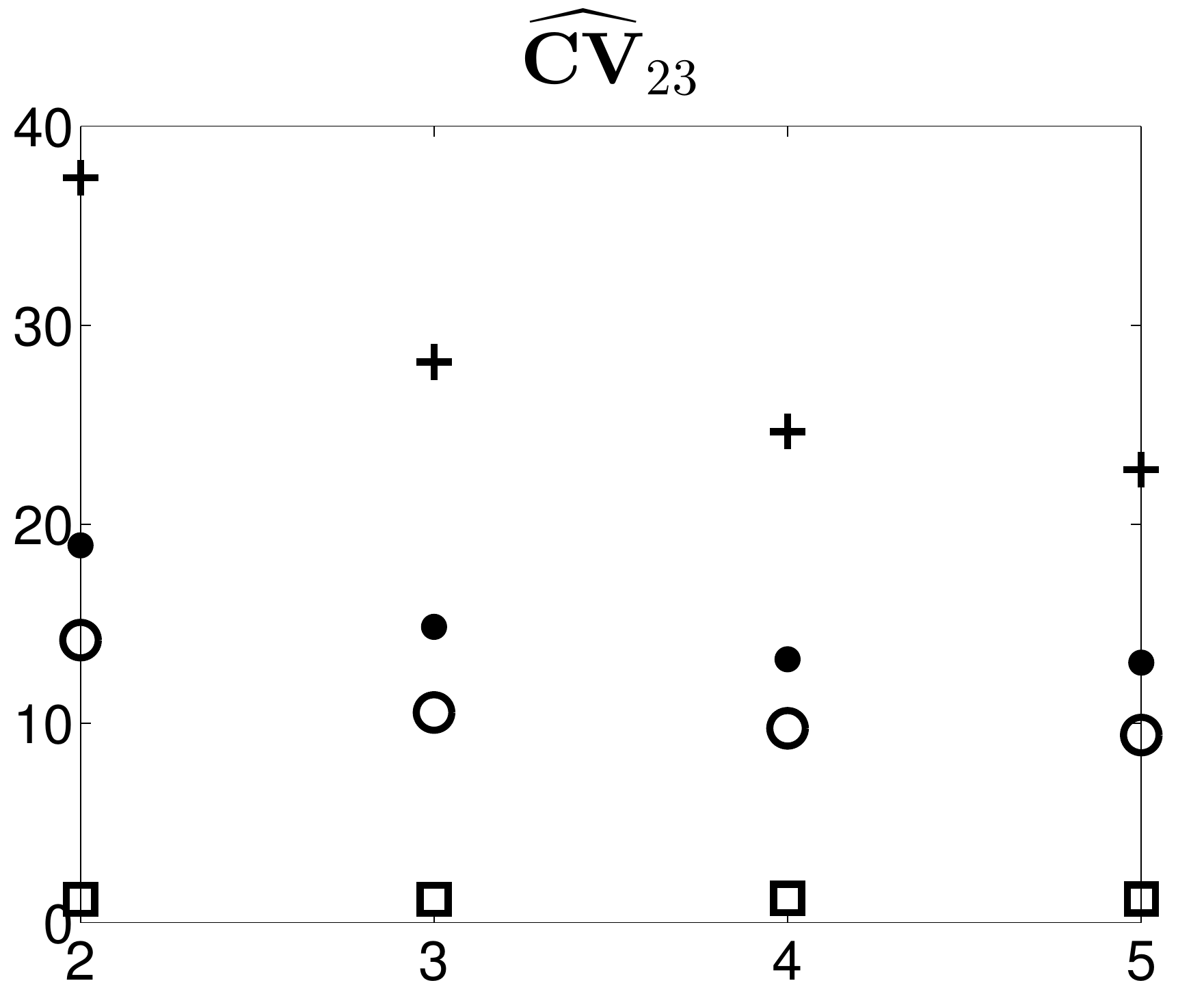}
\end{minipage}
\caption{The estimated coefficients of variation $\widehat{CV}_{ij}$ of the estimators of $\Phi_{2,0,2}(K_l)(e_i,e_j)$ plotted against $l$ for $\nobreak{i,j \in \{1,2,3\}}$. The CV of the estimator \eqref{IURn3} based on \textit{one} line is designated by ``$+$'', while the CV of the corresponding estimator based on \textit{three} lines is designated by ``$\bullet$''. The CV of the projection estimator is designated by ``$\circ$'' and ``$\square$'' for one and three lines, respectively.}
\label{4variances}
\end{figure}

The above example shows that only the projection estimator based on three orthogonal test lines has a satisfactory precision. For $l=2$ the CV's are approximately $\frac{1}{3}$ for the diagonal-elements and $1$ for the off-diagonal elements. Further variance reduction of the projection estimator can be obtained by using a larger number of systematic random test directions. For $n=2$ this can be effectuated by choosing equidistant points on the upper half circle; see \eqref{syst}. For $n=3$ the directions must be chosen evenly spread; see \cite{Leopardi2006} for details. 

If the projections are not available or too costly to obtain, systematic sampling in the position of the test lines with given orientations can be applied. In $\R^2$ this corresponds to a Steinhaus-type estimation procedure (see e.g. \cite{Jensen2005}). In $\R^3$ the fakir method described in \cite{L.1998} can be applied.

\subsection{Estimation based on vertical sections}\label{sec VUR}
In the previous section we constructed an estimator of $\T[s]$ based on isotropic uniform random lines. As described in \cite{Markus}, it is sometimes inconvenient or impossible to use the \textit{IUR} design in applications. For instance, in biology when analysing skin tissue, it might be necessary to use sample sections, which are normal to the surface of the skin, so that the different layers become clearly distinguishable in the sample. 
Instead of using \textit{IUR} lines it is then a possibility to use vertical sections introduced by Baddeley in \cite{Baddeley1983}. The idea is to fix a direction (the normal of the skin surface), and only consider flats parallel to this direction. After randomly selecting a flat among these flats, we want to pick a line in the flat in such a way that this line is an isotropic uniform random line in $\R^n$. Like in the classical formulae for vertical sections, we select this line in a non-uniform way according to a Blaschke-Petkantschin formula (see \eqref{B-P}). This idea is used to deduce estimators of $\T[s]$ from the Crofton formula \eqref{thm2}. 

When introducing the concept of vertical sections we use the following notation. For $ 0 \leq k \leq n $ and $ L \in \cL^n_k $, let
\begin{equation*}
\cLL=
\begin{cases}
\{M \in \cL_r^n \mid M \subseteq L \} & \text{if } 0 \leq r \leq k \\
\{M \in \cL_r^n \mid L \subseteq M \} & \text{if } k < r \leq n,
\end{cases}
\end{equation*}
and, similarly, let $\cE^E_r= \{F \in \En_r \mid F \subseteq E \} $ for $ E \in \cE^n_k  $ and $ 0 \leq r \leq k$. Let $  \nu^L_r$ denote the unique rotation invariant probability measure on $ \cL^L_r $, and let $ \mu^E_r $ denote the motion invariant measure on $ \cE^E_r $ normalized as in \cite{Weil}. 

Let $L_{0} \in \cLo$ be fixed. This is the \textit{vertical axis} (the normal of the skin surface in the example above). Let the reference set $A \subseteq \R^n$ be a compact set. 

\begin{definition}\label{VUR}
Let $1 < k < n$. A random $k$-flat $H$ in $\R^n$ is called \textbf{a vertical uniform random (VUR) $k$-flat hitting $A$} if the distribution of $H$ is given by 
\begin{equation*}
P(H \in \A) = c_2(A) \int_{\cLLn[k]} \int_{A \mid L^\perp} \1(L+x \in \A)\, \lambda_{\Lp}(dx) \, \nu_k^{L_0}(dL)
\end{equation*}
for $\A \in \B(\En_k)$, where $c_2(A)>0$ is a normalizing constant.
\end{definition}
The distribution of $H$ is concentrated on the set 
\begin{equation*}
\{E \in \En_{k} \mid E \cap A \neq \emptyset, L_0 \subseteq \Eorigo \}.
\end{equation*}
When the reference set $A$ is a convex body, the normalizing constant becomes
\begin{equation*}
c_2(A)=\binom{n-1}{k-1}\frac{\kappa_{n-1}}{\kappa_{k-1}\kappa_{n-k}}\frac{1}{V_{n-k}(A\vert L_0^\perp)}.
\end{equation*}
(Note that we do not indicate the dependence of $c_2(A)$ on $k$ by our notation.)  
This can be shown, e.g., by using the definition of $\nu_k^{L_0}$ together with \cite[(13.13)]{Weil}, Crofton's formula in the space $L_0^{\perp}$, and 
the equality
\begin{equation}\label{v0 lighed}
\1_{A \vert L^\perp}(x)=V_0((A\vert L_0^\perp) \cap (x+L))
\end{equation}
for $A \in \K^n$, $L \in \cLLn[k]$ and $x \in L^\perp$. For later use note that when $k=2$ the normalizing constant becomes 
\begin{equation} \label{normalizing constant}
c_2(A)=\frac{\omega_{n-1}}{2 \kappa_{n-2}V_{n-2}(A \vert L_0^\perp)}.
\end{equation}

To construct an estimator, which is based on a vertical uniform random flat, we cannot use Theorem \ref{crofton2} immediately as in the \textit{IUR}-case. It is necessary to use a Blaschke-Petkantschin formula first; see \cite[(2.8)]{Markus}. It states that for a fixed $L_0 \in \cLo$ and an integrable function $f \colon \cE^n_1 \rightarrow \R$, we have
\begin{align} \label{B-P}
\int_{\En_1} f(E) \, \mu_1^n(dE)& =  \frac{\pi \omega_{n-1}}{\omega_n} \int _{\cL^{L_0}_2} \int_{M^\perp} \int _{\cE_1^{M+x}} 
f(E) \sin (\angle(E,L_0))^{n-2} \nonumber
\\
&\qquad\qquad\times \mu_1^{M+x}(dE) \, \lambda_{M^{\perp}}(dx) \, \nu_2^{L_0} (dM),
\end{align}
where $\angle(E_1,E_2)$  is the (smaller) angle between $\pi(E_1)$ and $ \pi(E_2)$ for two lines $E_1,E_2 \in \cE^n_1$.
For $K \in \K^n$ and even $s \in \N_0$, equation \eqref{B-P} can be applied coordinate-wise to the mapping $E \mapsto \MT[s]$ and combined with the Crofton formula in Theorem \ref{crofton-like}. The result is an integral formula for two-dimensional vertical sections.

\begin{theorem}\label{VUR crofton}
Let $L_0 \in \cLo$ be fixed. If $K \in \K^n$ and $s \in \N_0$ is even, then 
\begin{align}\label{VUR crofton formel}
&\int _{\cL^{L_0}_2} \int_{M^\perp} \int _{\cE_1^{M+x}} \Phi^{(E)}_{0,0,s}(K \cap E) \sin(\angle(E,L_0))^{n-2} \, \mu_{1}^{M+x} (dE) \, \lambda_{M^{\perp}} (dx) \, \nu_{2}^{L_0} (dM)
\nonumber \\
&\qquad\qquad = \frac{2 \omega_{n+s+1}}{ s! \pi^2 \omega_{n-1} \omega_{s+1}^2 } \sum_{k=0}^{\frac{s}{2}} c_k^{(\frac{s}{2})} Q^{\frac{s}{2}-k}  \Phi_{n-1,0,2k}(K),
\end{align}
where the constants $c_k^{(m)}$ are given in Theorem $\ref{crofton-like}$. 
For odd $s \in \N_0$ the integral on the left-hand side is zero.
\end{theorem}

If Theorem \ref{crofton-like} is replaced by Theorem \ref{crofton2} in the above line of arguments, we obtain an explicit measurement function for vertical sections leading to one single tensor.

\begin{theorem}\label{VURcrofton2}
Let $L_0 \in \cLo$ be fixed. If $K \in \K^n$ and $s\in \N_0$ is even, then
\begin{align*}
\frac{\omega_n}{\pi \omega_{n-1}}\T[s]&=\int _{\cL^{L_0}_2} \int_{M^\perp} \int _{\cE_1^{M+x}} G_s(\Eorigo) V_0(K \cap E)
\\
& \qquad   \times  \sin(\angle(E,L_0))^{n-2} \, \mu_1^{M+x}(dE) \, \lambda_{M^{\perp}}(dx) \, \nu_2^{L_0}(dM),
\end{align*}
where $G_s$ is given in Theorem $\ref{crofton2}$.
\end{theorem}
Let $s \in \N_0$ be even and assume that $K \in \K^n$ is contained in a reference set $A \in \K^n$. Using Theorem \ref{VURcrofton2} we are able to construct unbiased estimators of the tensors $\T[s]$ of $K$ based on a vertical uniform random 2-flat.  If $H$ is an \textit{VUR} 2-flat hitting $A$ with vertical direction $L_0 \in \cLo$, then it follows from Theorem \ref{VURcrofton2} and \eqref{normalizing constant} that
\begin{equation}
V_{n-2}(A \vert L_0^\perp)
 \int_{\cE_1^{H}} G_s(\Eorigo)V_0(K \cap E)\sin(\angle(E,L_0))^{n-2}\, \mu_1^H(dE)
\end{equation} 
is an unbiased estimator of $\T[s]$. Hence the surface tensors can be estimated by a two-step procedure. First, let $H$ be a \textit{VUR} 2-flat hitting the convex body $A$ with vertical direction $L_0$. Given $H$, the integral
\begin{equation} \label{integral}
\int_{\cE_1^{H}} G_s(\Eorigo)V_0(K \cap E)\sin(\angle(E,L_0))^{n-2} \,\mu_1^H(dE)
\end{equation}
is estimated in the following way. Let $E \in \cE^H_1$ be an \textit{IUR} line in $H$ hitting $A$, i.e. the distribution of $E$ is given by 
\begin{equation*}
P(E \in \A)= c_3(A) \int_{\A} \1(A \cap E \neq \emptyset) \, \mu_1^H (dE), \qquad \A \in \B(\cE_1^H),
\end{equation*}
where
\begin{equation*}
c_3(A)=\frac{\pi}{2}V_{1}(A \cap H)^{-1}
\end{equation*}
is the normalizing constant. The integral \eqref{integral} is then estimated unbiasedly by 
\begin{equation}\label{Integral estimator}
c_3(A)^{-1}G_s(\Eorigo)V_0(K \cap E) \sin(\angle(E,L_0))^{n-2}.
\end{equation}
\begin{example}\rm 
Consider the case $s=2$. Let $H$ be a \textit{VUR} 2-flat hitting $A \in \K^n$ with vertical direction $L_0$. Given $H$, let $E$ be an \textit{IUR} line in $H$ hitting $A$. Then
\begin{equation*}
\frac{ \kappa_{n-2}V_{n-2}(A \vert L_0^\perp)V_1(A \cap H)}{ \omega_{n+1} }\bigg((n+1)Q(\Eorigo)-Q \bigg)V_0(K \cap E) \sin(\angle(E,L_0))^{n-2}
\end{equation*}
is an unbiased estimator of $\T[2]$.
\end{example}
   
Using \cite[(13.13)]{Weil} and an invariance argument, the integral \eqref{integral} can alternatively be expressed by means of the support function of $K$ in the following way
\begin{align*}
&\quad ~ \int_{\cE_1^{H}} G_s(\Eorigo)V_0(K \cap E)\sin(\angle(E,L_0))^{n-2} \,\mu_1^H (dE)
\\
&=\frac{1}{\omega_2}\int_{S^{n-1} \cap \Horigo} G_s(u^\perp \cap \Horigo)\sin(\angle(u^\perp \cap \Horigo,L_0))^{n-2} 
\\ & \qquad \qquad \quad \times \int_{[u]} V_0(K \cap H \cap (u^\perp + x)) \, \lambda_{[u]}(dx) \, \Haus^{1}(du) \nonumber
\\
&=\frac{1}{\omega_2}\int_{S^{n-1} \cap \Horigo}  G_s(u^\perp \cap \Horigo) \cos(\angle(u,L_0))^{n-2} w(K \cap H, u) \,\Haus^{1}(du),
\end{align*}
where $ [u] $ denotes the linear hull of a unit vector $ u $, and
\begin{equation*}
w(M,u)=h(M,u)+h(M,-u)
\end{equation*}
is the width of $M \in \K^n$ in direction $u$.
Hence, given $H$, 
\begin{equation} \label{Alternative integral estimator}
G_s(U^\perp \cap \Horigo)\cos(\angle(U,L_0))^{n-2} w(K \cap H, U)
\end{equation}
is an unbiased estimator of the integral \eqref{integral} if $U$ is uniform on $S^{n-1} \cap \Horigo$. 
As in the \textit{IUR} set-up in Section \ref{IUR} we have two estimators: an estimator \eqref{Integral estimator}, where it is only necessary to observe whether the random line $E$ hits or misses $K$, and the alternative estimator $\eqref{Alternative integral estimator}$, which requires more information. The latter estimator has a better precision at least when the reference set $A$ is large. Variance reduction can be obtained by combining the estimators with a systematic sampling approach.

\subsection{Estimation based on non-isotropic random lines}\label{Sec noniso}
In this section we consider estimators based on non-isotropic random lines. It is well-known from the theory of importance sampling, that variance reduction of estimators can be obtained by modifying the sampling distribution in a suitable way (see, e.g., \cite{Asmussen}). The estimators in this section are developed with inspiration from this theory. Let again $K \in \K^n$, and let $f \colon \cLo \rightarrow [0,\infty)$ be a density with respect to the invariant measure $\nu_1^n$ on $\cLo$ such that $f$ is positive $\nu^n_1$-almost surely. Then by Theorem \ref{crofton2} we have trivially
\begin{equation}\label{density integral}
\int_{\En_1} \frac{G_s(\Eorigo)V_0(K \cap E)}{f(\Eorigo)} \, f(\Eorigo)\,\mu_1^n(dE) = \T[s].
\end{equation}
Let $A \subseteq \R^n$ be a compact reference set containing $K$, and let $E$ be an $f$-weighted random line in $\R^n$ hitting A, that is, the distribution of $E$ is given by
\begin{align*}
P(E \in \A) = c_4(A) \int_{\A} \1(E \cap A \neq \emptyset) f(\Eorigo) \,  \mu_1^n (dE)
\end{align*}
for $\A \in \B(\En_1)$, where 
\begin{equation*}
c_4(A)=\bigg(\int_{\En_1} \1(E \cap A \neq \emptyset) f(\Eorigo)\, \mu_1^n (dE)\bigg)^{-1}
\end{equation*} 
is a normalizing constant. 
Then 
\begin{equation*}
\frac{c_4(A)^{-1} G_{s}(\Eorigo)V_0(K \cap E)}{f(\Eorigo)}
\end{equation*}
is an unbiased estimator of $\T[s]$. Notice that if we let the density $ f $ be constant, then this procedure coincides with the \textit{IUR} design in Section \ref{IUR}. 

Our aim is to decide, which density $f$ should be used in order to decrease the variance of the estimator of $\T[s]$. Furthermore, we want to compare this variance with the variance of the estimator based on an \textit{IUR} line. From now on, we restrict the investigation to the situation where $ n=2 $ and $ s=2 $. Furthermore, we assume that the reference set $ A $ is a ball in $ \R^2 $ of radius $ R $ for some $ R > 0 $. Then $ c_4(A)=(2R)^{-1} $ independently of $f$.

Since $ \Phi_{1,0,2}(K) $ can be identified with a symmetric $ 2 \times 2 $ matrix, we have to estimate three unknown components. We consider the variances of the three estimators separately.  
The components of the associated matrix of $G_2(L)$ for $L \in \cL^n_1$ is defined by
\begin{equation}
g_{ij}(L)=G_2(L)(e_i,e_j),
\end{equation}
for $i,j=1,2$, where $(e_1,e_2)$ is the standard basis of $\R^2$. More explicitly, by Example \ref{s er 4}, the associated matrix of $G_2(L)$ of the line $L=[u]$, for $u \in S^1$, is 
\begin{equation*}
\{g_{ij}([u])\}_{ij}=
\frac{3}{8}
\begin{pmatrix}
u_1^2 - \frac{1}{3} & u_1 u_2 \\
u_1u_2 & u_2^2-\frac{1}{3}
\end{pmatrix}.
\end{equation*} 
Now let 
\begin{equation*}
\ch := 2R \, g_{ij}(\Eorigo)V_0(K \cap E).
\end{equation*}
Then
\begin{equation}\label{estnon}
\frac{\ch}{f(\Eorigo)}
\end{equation}
is an unbiased estimator of $ \Phi_{1,0,2}(K)(e_i,e_j) $, when $ E $ is an $ f $-weighted random line in $ \R^2 $ hitting $ A $.

For a given $K \in \K^2$ the weight function $f$ minimizing the variance of the estimators of the form \eqref{estnon} can be determined.

\begin{lemma} \label{lemmaopt}
For a fixed $K \in \K^2$ with $\dim K \geq 1$ and $i,j \in \{1,2\}$, the estimator \eqref{estnon} has minimal variance if and only if $f=f_K^*$ holds $\nu^2_1-a.s.$, where
\begin{equation}
f_K^*(L)\propto \sqrt{2RV_1(K\vert L^{\perp})}\, \vert g_{ij}(L) \vert
\end{equation}
is a density with respect to $\nu^2_1$ that depends on $i,j$ and $K$.
\end{lemma}

\begin{proof}
As $K$ is compact, $f^*_K$ is a well-defined probability density, and since $\dim K  \geq 1$, the density $f^*_K$ is non-vanishing $\nu^2_1$-almost surely.  
The second moment of the estimator \eqref{estnon} is
\begin{equation} \label{2momentberegning}
\BE_f \bigg(\frac{\ch}{f(\Eorigo)}\bigg)^2 
= 2R \int_{\cL^2_1}V_1(K \vert L^\perp) \frac{g_{ij}(L)^2}{f(L)}\, \nu_1^2(dL),
\end{equation}
where $\BE_f$ denotes expectation with respect to the distribution of an $f$-weighted random line in $\R^2$ hitting $A$. The right-hand side of \eqref{2momentberegning} is the second moment of the random variable 
\begin{equation*}
\frac{\sqrt{2RV_1(K \vert \Lp)}\,g_{ij}(L)}{f(L)},
\end{equation*} 
where the distribution of the random line $L$ has density $f$ with respect to $\nu^2_1$. By \cite[Chapter 5, Theorem 1.2]{Asmussen} the second moment of this variable is minimized, when $f$ is proportional to $\sqrt{2RV_1(K \vert \Lp)}\, |g_{ij}(L)|$. Since the proof of \cite[Chapter 5, Theorem 1.2]{Asmussen} follows simply by an application of Jensen's inequality to the function $t \mapsto t^2$, equality can be characterized due to the strict convexity of this function, (see, e.g., \cite[(B.4)]{Gardner}). Equality holds if and only if $\sqrt{2RV_1(K \vert \Lp)}\,|g_{ij}(L)|$ is a constant multiple of $f(L)$ (or equivalently $f = f_K^*$) almost surely.  
\end{proof}

The proof of Lemma \ref{lemmaopt} generalizes directly to arbitrary dimension $n$. As a consequence of Lemma \ref{lemmaopt}, we obtain that for any convex body $K \in \K^2$, optimal non-isotropic sampling provides a strictly smaller variance of the estimator \eqref{estnon} than isotropic sampling. Indeed, noting that \eqref{estnon} with a constant function $f$ reduces to the usual estimator \eqref{s er 2} (with $n=2$, $A=RB^2$) based on \textit{IUR} lines, this follows from the fact that $f^*_K$ cannot be constant. If $f^*_K$ was constant almost surely, then $V_1(K \vert u^\perp) \propto |g_{ij}([u])|^{-2} $ for almost all $u \in S^1$. The left-hand side is essentially bounded, whereas the right-hand side is not. This is a contradiction.

A further consequence of Lemma \ref{lemmaopt} is that there does not exist an estimator of the form \eqref{estnon} independent of $K$ that has uniformly minimal variance for all $K \in \K^2$ with $\dim K \geq 1$. Unfortunately, $ f_K^* $ is not accessible, as it depends on $ K $, which is typically unknown. Even though estimators of the form \eqref{estnon} cannot have uniformly minimal variance for all $K \in \K^2$ with $\dim K \geq 1$, we now show that there is a non-isotropic sampling design which always yields smaller variance than the isotropic sampling design. 
Let
\begin{equation*}
f^*(L) \propto |g_{ij}(L)| 
\end{equation*}
be a density with respect to $\nu^2_1$. As $ |g_{ij}(L)| $ is bounded and non-vanishing for $ \nu^2_1$-almost all $ L $, $ f^* $ is well-defined and non-zero $ \nu^2_1$-almost everywhere. For convex bodies of constant width, the density $f^*$ coincides with the optimal density $f^*_K$.

\begin{theorem}\label{onecomponent}
Let $ K \in \K^2 $, and let $ A = RB^2 $ for some $ R>0 $ be such that $ K \subseteq A $. Then
\begin{equation}\label{varineq}
\text{Var}_{f^*}\bigg(\frac{\ch}{f^*(\Eorigo)}\bigg)
<
\text{Var}_{IUR}\big(\ch \big).
\end{equation}
\end{theorem}

\begin{proof}
Using the fact that both estimators are unbiased, it is sufficient to show that there is a $0 < \lambda < 1$ with
\begin{equation}\label{Eineq}
\BE_{f^*}\bigg(\frac{\ch}{f^*(\Eorigo)}\bigg)^2 \leq
\lambda \, \BE_{IUR}\big(\ch\big)^2,
\end{equation}
for all $K \in \K^2$. Using \eqref{2momentberegning}, the left-hand side of this inequality is 
\begin{equation*}
2R \int_{\cL^2_1}|g_{ij}(L)| \, \nu_1^2 (dL) \int_{\cL^2_1} |g_{ij}(L)|V_{1}(K \vert L^\perp) \, \nu_1^2 (dL)
\end{equation*}
and the right-hand side is
\begin{equation*}
2R \int_{\cL^2_1} g_{ij}(L)^2 \, V_{1}(K  \vert L^\perp) \, \nu_1^2 (dL).
\end{equation*} 

Since $u \mapsto V_1(K \vert u^{\perp})$ is the support function of an origin-symmetric zonoid, the inequality \eqref{Eineq} holds if
\begin{align}\label{support ineq}
&  \int_0^{2 \pi}|g_{ij}([u_{\phi}])| \, \frac{d\phi}{2\pi} \, \int_0^{2\pi} |g_{ij}([u_{\phi}])|\, h(Z,u_{\phi}) \, \frac{d\phi}{2\pi}  \nonumber
\\ 
& \quad \leq \lambda \, \int_0^{2 \pi} g_{ij}([u_{\phi}])^2 h(Z,u_{\phi}) \, \frac{d\phi}{2\pi}
\end{align}
for any origin-symmetric zonoid $ Z $. Here $ u_{\phi} = (\cos(\phi),\sin(\phi))^\top $ for $ \phi \in [0,2\pi] $.
As support functions of zonoids can be uniformly approximated by support functions of zonotopes (see, e.g., \cite[Theorem 1.8.14]{Schneider93}) and the integrals in \eqref{support ineq}  depend linearly on these support functions, it is sufficient to show \eqref{support ineq} for all origin-symmetric line segments $Z$ of length two. Hence, we may assume that $Z$ is an origin-symmetric line segment with endpoints $\pm(\cos(\gamma), \sin(\gamma))^\top$, where $\gamma \in [0,\pi)$. We now substitute the support function
\begin{equation*}
h(Z,u_{\phi})=|\cos(\phi-\gamma)|
\end{equation*}
for $\phi \in [0,2\pi),$ into \eqref{support ineq}.

First, we consider the estimation of the first diagonal element of $\Phi_{1,0,2}(K)$, that is, $i,j=1$ and $g_{ij}([u_\phi])= \frac{3}{8}(\cos^2(\phi)-\frac{1}{3}) $ for $ \phi \in [0,2\pi] $. The integrals in \eqref{support ineq} then become
\begin{equation*}
P_{f^*}(\gamma):= \frac{3}{8}\int_0^{2\pi}|\cos^2(\phi)-\frac{1}{3}| \, \frac{d\phi}{2 \pi} \, \frac{3}{8} \int_{0}^{2\pi}|\cos^2(\phi)-\frac{1}{3}| |\cos(\phi-\gamma)| \, \frac{d\phi}{2\pi}
\end{equation*}
and
\begin{equation*}
P_{IUR}(\gamma):=\frac{9}{64}\int_0^{2\pi}\bigg(\cos^2(\phi)-\frac{1}{3}\bigg)^2 |\cos(\phi- \gamma)| \, \frac{d\phi}{2 \pi}.
\end{equation*}
Let $ \kappa = \arccos(\frac{1}{\sqrt{3}}) $. Then
\begin{equation*}
M:= \frac{3}{8}\int_0^{2\pi}|\cos^2(\phi)-\frac{1}{3}| \, \frac{d\phi}{2 \pi} 
=\frac{\sqrt{2}+\kappa}{4\pi}-\frac{1}{16},
\end{equation*}
and elementary, but tedious calculations show that
\begin{align*}
P_{f^*}(\gamma)&=\frac{M}{ \pi} \bigg(\frac{2\sqrt{2}}{3\sqrt{3}}\cos(\gamma) - \frac{1}{4}\cos^2(\gamma) \bigg)\1_{[0,\frac{\pi}{2}-\kappa]}(\gamma)
\\
&\qquad +
\frac{M}{\pi}\bigg(\frac{1}{4}\cos^2(\gamma) + \frac{1}{3\sqrt{3}}\sin(\gamma) \bigg)\1_{(\frac{\pi}{2}-\kappa,\frac{\pi}{2}]}(\gamma)
\end{align*}
for $ \gamma \in [0,\frac{\pi}{2}] $. Further,  $P_{f^*}(\gamma)=P_{f^*}(\pi-\gamma) $ for $\gamma \in [\frac{\pi}{2},\pi] $. For the \textit{IUR} estimator we get that
\begin{equation*}
P_{IUR}(\gamma)=\frac{1}{20\pi}\bigg(-\frac{3}{8}\cos^4(\gamma) + \cos^2(\gamma) + \frac{1}{2} \bigg)
\end{equation*}
for $ \gamma \in [0,\frac{\pi}{2}] $, and $ P_{IUR}(\gamma)=P_{IUR}(\pi -\gamma) $ for $ \gamma \in [\frac{\pi}{2},\pi] $. The functions $ P_{f^*} $ and $ P_{IUR} $ are plotted in Figure \ref{Secondmoments}. Basic calculus for the comparison of these two functions shows that $P_{f^*} < P_{IUR}$. This implies that $P_{f^*} \leq \lambda P_{IUR}$, where $\nobreak{\lambda=\max_{\gamma \in [0,\pi]} \frac{P_{f^*(\gamma)}}{P_{IUR}(\gamma)}}$ is smaller than one as $P_{f^*}$ and $P_{IUR}$ are continuous on the compact interval $[0,\pi]$. Hereby \eqref{support ineq} is satisfied for $i=j=1$. Interchanging the roles of the coordinate axes in \eqref{support ineq} yields the same result for $i=j=2$.

We now consider estimation of the off-diagonal element, that is, $i=1$, $j=2$. Then the left-hand and the right-hand side of \eqref{support ineq} become
\begin{equation}\label{Qfs}
Q_{f^*}(\gamma)=\frac{3}{8} \int_0^{2\pi}|\cos(\phi)\sin(\phi)| \, \frac{d\phi}{2 \pi}
\,\frac{3}{8} \int_{0}^{2\pi} |\cos(\phi)\sin(\phi)| |\cos(\phi-\gamma)| \, \frac{d\phi}{2 \pi} 
\end{equation}
and
\begin{equation}\label{M}
Q_{IUR}(\gamma)=\frac{9}{64}\int_0^{2\pi}\cos^2(\phi)\sin^2(\phi) |\cos(\phi- \gamma)| \,\frac{d\phi}{2 \pi}
\end{equation}
for $ \gamma \in [0,\pi] $. 
We have 
\begin{equation*}
\frac{3}{8} \int_0^{2\pi}|\cos(\phi)\sin(\phi)| \, \frac{d\phi}{2 \pi} = \frac{3}{8\pi},
\end{equation*}
and then
\begin{equation*}
Q_{f^*}(\gamma) = 
\frac{3}{32\pi^2}\bigg(\sin(\gamma)+\cos(\gamma)-\sin(\gamma)\cos(\gamma) \bigg)
\end{equation*}
for $ \gamma \in [0,\frac{\pi}{2}] $, and $ Q_{f^*}(\gamma)=Q_{f^*}(\gamma-\frac{\pi}{2}) $ for $ \gamma \in [\frac{\pi}{2},\pi] $. 
For $ \gamma \in [0,\pi]$ we further find that
\begin{equation*}
Q_{IUR}(\gamma) = \frac{3}{320 \pi}\bigg( 4-\frac{1}{2}\sin^2(2\gamma)  \bigg).
\end{equation*}
The functions $ Q_{IUR} $ and $ Q_{f^*} $ are plotted in Figure \ref{Secondmoments2}. Basic calculus shows that
\begin{equation}\label{minmaxf}
\min_{0 \leq \gamma \leq \pi} Q_{f^*} = \frac{3}{32 \pi^2}\left(\sqrt{2}-\frac{1}{2}\right),
\qquad 
\max_{0 \leq \gamma \leq \pi}Q_{f^*}=\frac{3}{32 \pi^2},
\end{equation}
and
\begin{equation}\label{maxminIUR}
\min_{0 \leq \gamma \leq \pi} Q_{IUR} = \frac{21}{640 \pi},
\qquad
\max_{0 \leq \gamma \leq \pi} Q_{IUR} = \frac{3}{80 \pi}.
\end{equation}
Hence 
\begin{equation*}
Q_{f^*}(\gamma) \leq \frac{3}{32 \pi^2} \le \lambda \frac{21}{640 \pi} \leq \lambda \, Q_{IUR}(\gamma)
\end{equation*}
for $\gamma \in [0,\pi]$ with $\lambda=\frac{3}{\pi}< 1$. Hereby \eqref{support ineq} holds for all zonotopes $Z$ and $i=1,j=2$, and the claim is shown.
\end{proof}

\begin{figure}
\begin{center}
\includegraphics[scale=0.5]{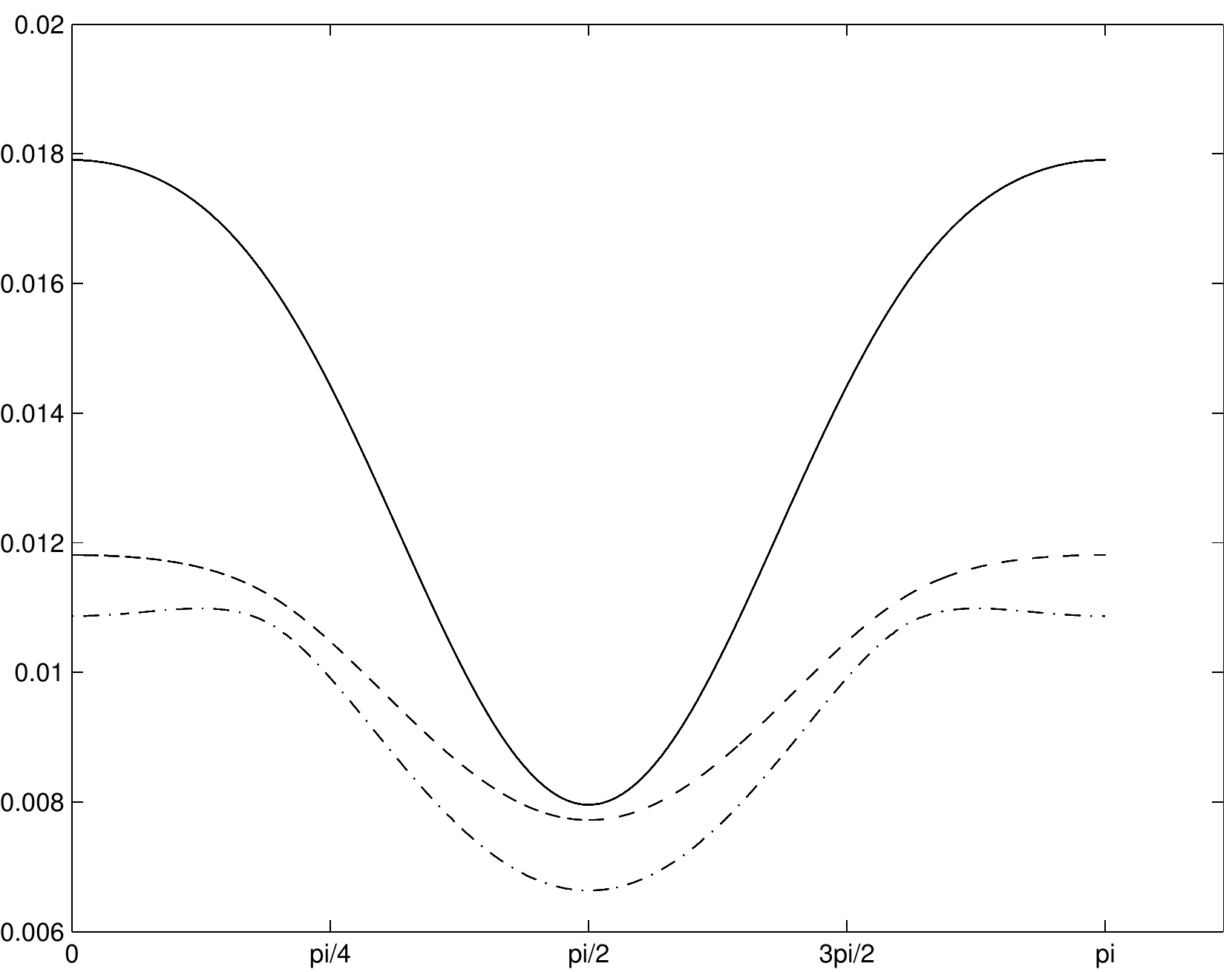}
\end{center}
\caption{The straight line is $ P_{IUR} $, the dashed line is $ P_{f^*} $, and the dash-dotted line is $ P_{opt} $.}
\label{Secondmoments}
\end{figure}

\begin{figure}
\begin{center}
\includegraphics[scale=0.5]{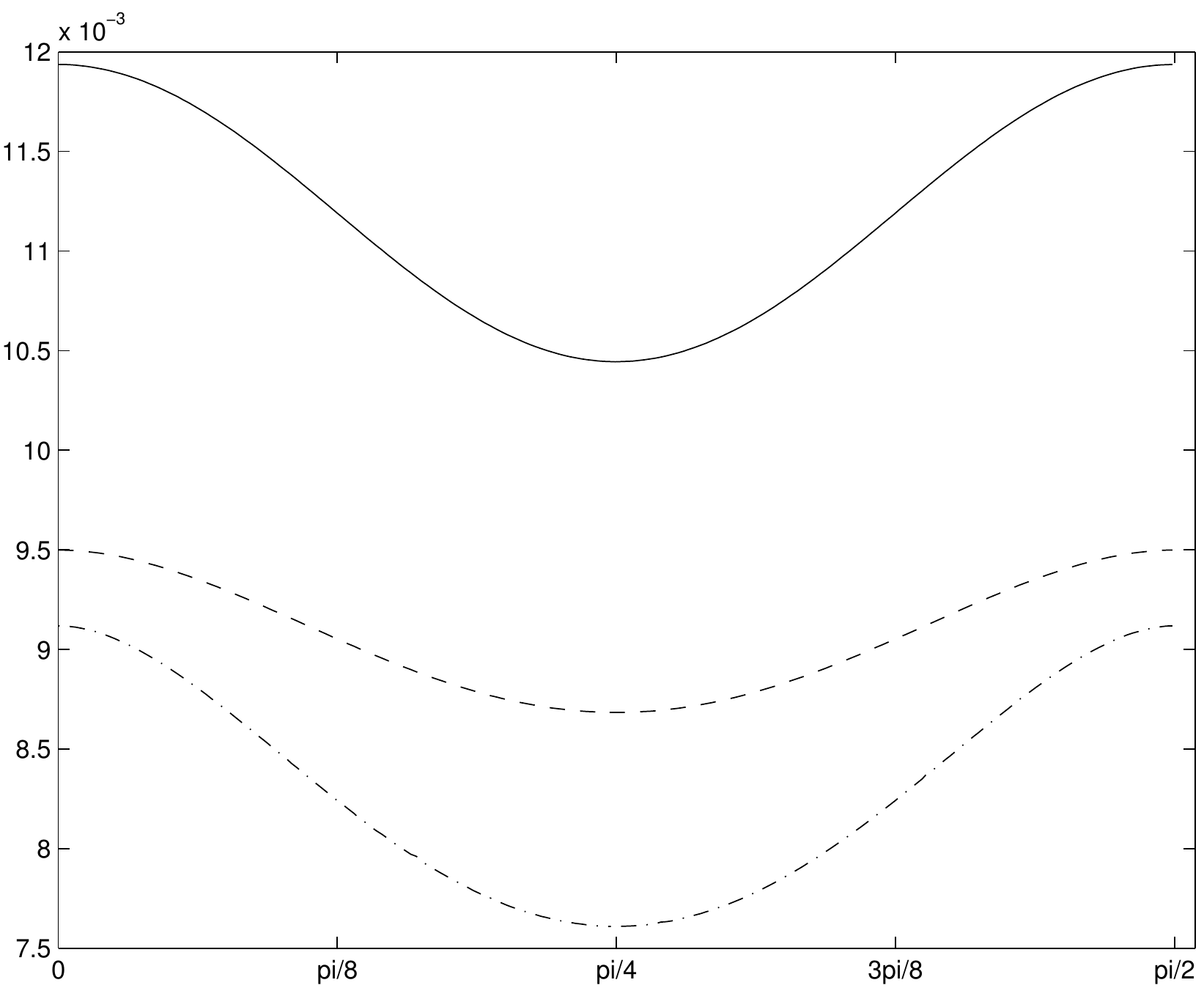}
\end{center}
\caption{The straight line is $ Q_{IUR} $, the dashed line is $ Q_{f^*} $, and the dash-dotted line is $ Q_{opt} $.}
\label{Secondmoments2}
\end{figure}

If $ E $ is an $ f^* $-weighted random line suited for estimating one particular component of $ \Phi_{1,0,2}(K) $, then $ E $ should not be used to estimate any of the other components, as this would increase the variance of these estimators considerably. Hence, if we estimate all of the components of the tensor using the estimator based on $ f^* $-weighted lines, we need three lines; one for each component. If we want to compare this approach with an estimation procedure based on \textit{IUR} lines, requiring the same workload, we will use \emph{three} \textit{IUR} lines. Note however, that all three \textit{IUR} lines can be used to estimate \emph{all} three components of the tensor. 
This implies that we should actually compare the variance of the estimator based on \emph{one} $ f^* $-weighted random line with the variance of an estimator based on \emph{three} \textit{IUR} lines. It turns out that the estimator based on \emph{three} independent \textit{IUR} lines has always smaller variance, than the estimator based on \emph{one} $ f$-weighted line, no matter how the density $f$ is chosen. 

\begin{theorem}
Let $ K \in \K^2 $, and let $ A =RB^2 $ with some $ R > 0 $ be such that $ K \subseteq A $. Let $ f $ be a density with respect to $ \nu^2_1 $, which is non-zero $ \nu^2_1 $-almost everywhere. Let $ E_1,E_2$ and $ E_3 $ be independent \textit{IUR} lines in $ \R^2 $ hitting $ A $. Then
\begin{equation*}
Var\bigg(\frac{1}{3}\sum_{k=1}^3 \hat{\varphi}_{ij}(K \cap E_k) \bigg)
<
Var_{f}\bigg(\frac{\ch}{f(\pi(E))}\bigg)
\end{equation*}
for $i,j \in \{1,2\}$.
\end{theorem}
\begin{proof}
By Theorem \ref{onecomponent}, the variance of the estimator \eqref{estnon} is bounded from below by the variance of the same estimator with $f=f_K^*$. Hence, it is sufficient to compare the second moments of
\begin{equation*}
\frac{1}{3}\sum_{k=1}^3 \hat{\varphi}_{ij}(K \cap E_k)
\end{equation*}
and \eqref{estnon} with $f=f_K^*$. The latter is
\begin{equation*}
2R \bigg(\int_{\cL^2_1} |g_{ij}(L)| \sqrt{V_1(K \vert L^\perp)} \, \nu^2_1 (dL) \bigg)^2,
\end{equation*}
so let
\begin{equation*}
P_{opt}(\gamma):= \bigg(\frac{3}{8} \int_0^{2\pi} |\cos^2(\phi)-\frac{1}{3}| \sqrt{|\cos(\phi - \gamma)|} \, \frac{d\phi}{2 \pi} \bigg)^2
\end{equation*}
and
\begin{equation*}
Q_{opt}(\gamma):= \bigg( \frac{3}{8} \int_0^{2\pi}|\cos(\phi)\sin(\phi)| \sqrt{|\cos(\phi - \gamma)|} \, \frac{d\phi}{2 \pi} \bigg)^2
\end{equation*}
for $ \gamma \in [0,\pi] $. Using the notation of the previous proofs, by \eqref{Qfs}, \eqref{M}, \eqref{minmaxf} and \eqref{maxminIUR} we have 
\begin{equation*}
Q_{opt}(\gamma) \geq \bigg(\frac{8 \pi Q_{f^*}(\gamma)}{3}\bigg)^2  \geq \frac{9-4\sqrt{2}}{64 \pi^2} > \frac{1}{80\pi} \geq \frac{1}{3}Q_{IUR}(\gamma)
\end{equation*}
for $\gamma \in [0,\pi]$.
Likewise, $ P_{opt}(\gamma) \geq \bigg(\frac{P_{f^*}(\gamma)}{M}\bigg)^2$. Elementary analysis shows that
\begin{equation*}
\min_{0\leq \gamma \leq \frac{\pi}{2}-\kappa}\bigg(\frac{P_{f^*}(\gamma)}{M}\bigg)^2=\frac{25}{324 \pi^2}
> \frac{3}{160\pi} = \max_{0\leq \gamma \leq \frac{\pi}{2} - \kappa} \frac{1}{3} P_{IUR}(\gamma),
\end{equation*}
and that
\begin{equation*}
\bigg(\frac{P_{f^*}(\gamma)}{M}\bigg)^2-\frac{1}{3}P_{IUR}(\gamma) 
\geq 
\bigg(\frac{P_{f^*}(\frac{\pi}{2})}{M}\bigg)^2-\frac{1}{3}P_{IUR}(\frac{\pi}{2}) > 0
\end{equation*}
on $[\frac{\pi}{2}-\kappa,\frac{\pi}{2}]$. Hence $P_{opt}>\frac{1}{3}P_{IUR}$ on $[0,\pi]$, and the assertion is proved.
\end{proof}

This leads to the following conclusion: If one single component of the tensor $\T[2]$ is to be estimated for unknown $K$, the estimator \eqref{estnon} with $f=f^*$ is recommended, as its variance is strictly smaller than the one from isotropic sampling (where $f$ is a constant). If all components are sought for, the estimator based on three \textit{IUR} lines should be preferred.

\section{Model based estimation} \label{SecModel}
In this section we derive estimators of the specific surface tensors associated with a stationary process of convex particles based on linear sections. In \cite{RSRS06}, Schneider and Schuster treat the similar problem of estimating the area moment tensor ($ s=2 $) associated with a stationary process of convex particles using planar sections.

Let $X$ be a stationary process of convex particles in $\R^n$ with locally finite (and non-zero) intensity measure, intensity  $\gamma >0$  and grain distribution $\Q$ on $\K_0:=\{K \in \K^n \mid c(K)=0\}$; see, e.g., \cite{Weil} for further information on this basic model of stochastic geometry. Here $c \colon \K^n \setminus \{\emptyset\} \rightarrow \R^n$ is the center of the circumball of $K$. Since $ X $ is a stationary process of convex particles, the intrinsic volumes $V_0, \dots, V_n$ are $ \Q $-integrable by \cite[Theorem 4.1.2]{Weil}. For $j \in \{0, \dots, n-1\}$ and $s \in \N_0$ the tensor valuation $\Phi_{j,0,s}$ is measurable and translation invariant on $\K^n$, and since, by \eqref{Mtensor},
\begin{equation*}
|\jT[s](e_{i_1}, \dots, e_{i_s})| \leq \frac{\omega_{n-j}}{s!\omega_{n-j+s}}V_{j}(K),
\end{equation*}
it is coordinate-wise $\Q$-integrable. The \textit{$j$th specific (translation invariant) tensor of rank s} can then be defined as
\begin{equation} \label{defmeansurface}
\joT[s]:=\gamma \int_{\K_0}\jT[s] \,\Q (dK)
\end{equation}
for $j \in \{0, \dots, n-1\}$ and $s \in \N_0$. For $j=n-1$, the specific tensors are called the specific surface tensors. Notice that $\oT[2]=\frac{1}{8 \pi} \overline{T}(X)$, where $\overline{T}(X)$ is the mean area moment tensor described in \cite{RSRS06}. By \cite[Theorem 4.1.3]{Weil} the specific tensors of $X$ can be represented  as
\begin{equation}\label{intuitive1}
\joT=\frac{1}{\lambda(B)} \; \BE \, \sum_{\mathclap{\substack{K \in X \\ c(K)\in B}}} \, \jT[s],
\end{equation}
where $B \in \B(\R^n)$ with $0 < \lambda(B) < \infty$. 

In the following we restrict to $j=n-1$ and discuss the estimation of $\oT[s]$ from linear sections of $X$. We assume from now on that $n\geq 2$. For $L \in \cLo$ we let $X \cap L :=\{K \cap L \mid K \in X, K \cap L \neq \emptyset\}$ be the stationary process of convex particles in $L$ induced by $X$. Let $\gamma_L$ and $\Q_L$ denote the intensity and the grain distribution of $ X \cap L $, respectively. The tensor valuation $\Phi_{0,0,s}^{(L)}$ is measurable and $\Q_L$-integrable on $ K_0^{(L)} :=\{K \in \K_0 \mid K \subseteq L\}$. We can thus define
\begin{equation*}
\oMT:=\gamma_L \int_{\K_0^{(L)}} \Phi_{0,0,s}^{(L)}(K) \, \Q_L (dK).
\end{equation*}
This deviates in the special case $\overline{T}^{(L)}(X \cap L) = 8 \pi \, \overline{\Phi}_{0,0,2}^{(L)}(X \cap L)$ from the definition in \cite{RSRS06} due to a misprint there.
An application of \eqref{Simpel form for T} yields,
\begin{equation}\label{simpel form, model}
\oMT= \frac{2}{s! \omega_{s+1}}Q(L)^{\frac{s}{2}} \gamma_L
\end{equation}
for even $s$, and $\oMT=0$ for odd $s$.

\begin{theorem} \label{modelthm}
Let $X$ be a stationary process of convex particles in $\R^n$ with positive intensity. If $s \in \N_0$ is even, then
\begin{equation}\label{model integral}
\int_{\cLo} \oMT \, \nu_1^n (dL)= \frac{2 \omega_{n+s+1}}{\pi s! \omega_{s+1}^2 \omega_n} \sum_{k=0}^{\frac{s}{2}}c_k^{(\frac{s}{2})}Q^{\frac{s}{2}-k}\,\oT[2k],
\end{equation}
where the constants $c_k^{(\frac{s}{2})}$ for $k=0, \dots, \frac{s}{2}$ are given in Theorem \ref{crofton-like}.
\end{theorem}

\begin{proof}
Let $L \in \cLo$, and let $\gamma_L$ be the intensity of the stationary process $X \cap L$. If $B \subseteq L$ is a Borel set with $\lambda_L(B)=1$, then an application of Campbell's theorem and Fubini's theorem yields
\begin{align*}
\gamma_L&=
\BE \; \sum_{\mathclap{\substack{K \in X \\ K \cap L \neq \emptyset}}} \; \1(c(K \cap L) \in B)
\\
&=\gamma \int_{\K_0} \int_{\Lp} V_0(K \cap (L + x)) \, \lambda_{\Lp}(dx) \, \Q(dK),
\end{align*}
where $\gamma$ and $\Q$ are the intensity and the grain distribution of $X$.
Then, \eqref{simpel form, model} implies that
\begin{equation*}
\oMT=\gamma \int_{\K_0} \int_{\Lp} \Phi_{0,0,s}^{(L+z)}(K \cap (L+z)) \, \lambda_{\Lp}(dz) \,\Q(dK),
\end{equation*}
and by Fubini's theorem we get
\begin{align}
\int_{\cLo} \oMT \, \nu^n_1(dL)
=\gamma \int_{\K_0} \int_{\En_1} \Phi^{(E)}_{0,0,s}(K \cap E) \,\mu_1^n(dE)\, \Q(dK). \label{Eq}
\end{align}
Now Theorem \ref{crofton-like} yields the stated integral formula \eqref{model integral}.
\end{proof}

A combination of equation \eqref{Eq} and equation \eqref{reffinal} immediately gives the following Theorem \ref{model2}, which suggests an estimation procedure of the specific surface tensor $\oT$ of the stationary particle process $X$.

\begin{theorem}\label{model2}
Let $X$ be a stationary process of convex particles in $\R^n$ with positive intensity. If $s \in \N_0$ is even, then
\begin{equation}\label{inversmodel}
\int_{\cLo} \sum_{j=0}^{\frac{s}{2}} d_{\frac{s}{2}\,j}C_{2j} Q^{\frac{s}{2}-j} \oMT[2j] \, \nu^n_1(dL)= \oT,
\end{equation}
where $d_{\frac{s}{2}\,j}$ and $ C_{2j} $ for $ j=0, \dots, \frac{s}{2} $ are given before Theorem \ref{crofton2}.
\end{theorem}
Using \eqref{simpel form, model}, we can reformulate the integral formula \eqref{inversmodel} in the form
\begin{equation*}
\int_{\cLo} G_s(L) \gamma_L \, \nu^n_1(dL)= \oT,
\end{equation*}
where $G_s$ is given in Theorem \ref{crofton2}.

\begin{example}\rm 
In the case where $ s=2 $ formula \eqref{inversmodel} becomes
\begin{equation*}
\int_{\cL_1^n}  \frac{2 \pi^2\omega_n}{\omega_{n+3}}\oMT[2]-\frac{\omega_n}{4 \omega_{n+1}}Q \oMT[0] \, \nu_1^n(dL)=\oT[2].
\end{equation*}
Up to a normalizing factor $ 2 \pi $ in the constant in front of $ \overline{\Phi}_{0,0,2}^{(L)} $, this formula coincides with formula (7) in \cite{RSRS06}, when $ n=2 $. Apparently the normalizing factor got lost, when Schneider and Schuster used \cite[(36)]{TVCB}, which is based on the spherical Lebesgue measure. In \cite{RSRS06}, Schneider and Schuster use the \emph{normalized} spherical Lebesgue measure.
\end{example}

\noindent \textbf{Acknowledgements} The authors acknowledge support by the German research foundation (DFG) through the research group ``Geometry and Physics of Spatial Random Systems'' under grants HU1874/2-1, HU1874/2-2 and by the Centre for Stochastic Geometry and Advanced Bioimaging, funded by a grant from The Villum Foundation.

\bibliography{litteratur}

\end{document}